\documentclass[10pt]{amsart}
\usepackage{amssymb}
\theoremstyle{plain}
\newtheorem{theorem}{Theorem}[section]

\newtheorem{proposition}[theorem]{Proposition}
\newtheorem{lemma}[theorem]{Lemma}
\newtheorem{corollary}[theorem]{Corollary}
\theoremstyle{remark}
\newtheorem{remark}[theorem]{Remark}

\theoremstyle{definition}

\theoremstyle{definition}

\newcommand{\N}{\mathbb{N}}

\newcommand{\Z}{\mathbb{Z}}
\newcommand{\R}{\mathbb{R}}
\newcommand{\cF}{\mathcal F}
\newcommand{\cT}{\mathcal T}
\newcommand{\cM}{\mathcal M}
\newcommand{\cI}{\mathcal I}

\newcommand{\vp}{\varepsilon}
\newcommand{\xb}{\overline{x}}
\newcommand{\lam}{\lambda}
\newcommand{\summ}{\sum_{l\in\Z\setminus\{0\}}}
\newcommand{\Alstar}{A_l^{{}^*}}
\newcommand{\lnorm}{\left\|}
\newcommand{\rnorm}{\right\|}

\newcommand\PW{ PW_\pi}
\newcommand\PWt{ PW^{(2)}_\pi}
\newcommand\Il{I_{{}_\lambda}}
\newcommand{\ft}[1]{{\widehat{#1}}}
\newcommand{\FT}[1]{{\cF[{#1}]}}
\newcommand{\FTd}[1]{{\cF^{(d)}[{#1}]}}
\newcommand{\FTt}[1]{{\cF^{(2)}[{#1}]}}
\newcommand{\LoneR}{L_1(\R)}
\newcommand{\LtwoR}{L_2(\R)}
\newcommand{\lint}[2]{{\int_{#1}^{#2}}}
\newcommand{\Top}[1]{{\cT}_{{}_{[{#1}]}}}
\newcommand{\Mop}[1]{{\cM}_{{}_{[{#1}]}}}
\newcommand{\Lwoodone}{2\pi}
\newcommand{\CR}{C(\R)}
\newcommand{\LtwoT}{L_2[-\pi,\pi]}
\newcommand{\gl}{g_\lambda}

\title{On the sampling and recovery of bandlimited functions 
via scattered translates of the Gaussian}

\author{Th. Schlumprecht}
\address{Department of Mathematics, Texas A\&M University 
College Station, TX 77843, USA}
\email{thomas.schlumprecht@math.tamu.edu}

\author{N.~Sivakumar}
\address{Department of Mathematics, Texas A\&M University 
College Station, TX 77843, USA}
\email{sivan@math.tamu.edu}

\thanks{The research of
the  first author was supported by the NSF}

\begin{document}
\maketitle
\markright{INTERPOLATION OF BANDLIMITED FUNCTIONS USING GAUSSIANS}
\begin{abstract}Let $\lam$ be a positive number, and let
$(x_j:j\in\Z)\subset\R$ be a fixed 
Riesz-basis sequence, namely,
$(x_j)$ is strictly increasing, and the
set of functions 
$\{\R\ni t\mapsto e^{ix_jt}:j\in\Z\}$
is a Riesz basis ({\it i.e.,\/} unconditional
basis) for $\LtwoT$.  Given a function
$f\in\LtwoR$ whose Fourier transform is 
zero almost everywhere outside the interval
$[-\pi,\pi]$, there is a unique
square-summable sequence $(a_j:j\in\Z)$, depending
on $\lambda$ and $f$, such that the function
\[
\Il(f)(x):=\sum_{j\in\Z}a_je^{-\lam(x-x_j)^2}, \qquad x\in\R,
\]
is continuous and square integrable on $(-\infty,\infty)$,
and satisfies the interpolatory conditions
$\Il(f)(x_j)=f(x_j)$,
$j\in\Z$.
It is shown that $\Il(f)$ converges
to $f$ in $\LtwoR$, and also uniformly on
$\R$, as $\lam\to0^+$.  A multidimensional
version of this result is also obtained.  
In addition, the fundamental
functions for the univariate interpolation
process are defined, and some of their basic
properties, including their exponential
decay for large argument, are established.
It is further shown that the 
associated interpolation operators
are bounded on $\ell_p(\Z)$ for 
every $p\in[1,\infty]$.
\end{abstract}
\section{Introduction}\label{S:0}

This paper, one in the long tradition of those involving
the interpolatory theory of functions, is concerned
with interpolation of data via the translates
of a Gaussian kernel.  The motivation for this work
is twofold.  The first is the theory
of {\it Cardinal Interpolation\/}, which deals
with the interpolation of data prescribed at
the integer lattice, by means of the integer
shifts of a single function.  This subject has a
rather long history, and it enjoys interesting connections
with other branches of 
pure and applied mathematics, {\it e.g.\/}
Toeplitz matrices, Function Theory, Harmonic Analysis,
Sampling Theory.
When the underlying function (whose shifts form
the basis for interpolation) is taken to be 
the so-called {\it Cardinal B-Spline\/}, one 
deals with {\it Cardinal Spline Interpolation\/},
a subject championed by Schoenberg, and taken 
up in earnest by a host of followers.  More recently,
it was discovered that there is a remarkable
analogy between cardinal spline interpolation
and cardinal interpolation by means of the (integer)
shifts of a Gaussian, a survey of which
topic may be found in \cite{RS3}.  The current
article may also be viewed as a contribution in
this vein;  it too explores further connections
between the interpolatory theory of splines and 
that of the Gauss kernel, but does so in the context of
interpolation at point sets which are more
general than the integer lattice.  This brings
us to the second, and principal, motivating 
influence for our work,
namely the researches of Lyubarskii and Madych
\cite{LM}.  This duo have considered 
spline interpolation at certain (infinite)
sets of points which are generalizations
of the integer lattice, and we were prompted
by their work to ask if the analogy between
splines and Gaussians, very much in evidence
in the context of cardinal interpolation, persists
in this `nonuniform' setting also.  Our paper
seeks to show that this is indeed the case.  The
influence of \cite{LM} on our work goes further.
Besides
providing us with the motivating question for 
our studies, it also offered us an array of 
basic tools which we have modified and adapted.

We shall supply more particulars -- of a technical
nature -- concerning the present paper later in this
introductory section, soon after we finish discussing
some requisite general material.

A basic tool in our analysis is the {\it Fourier
Transform\/}, so we assemble some basic and 
relevant facts about it here; our sources
for this material are \cite{Go} and 
\cite{Ch}.
If 
$g\in\LoneR$, then the Fourier transform
of $g$, $\widehat{g}$, is defined as follows:
\begin{equation}\label{E:wrongdef}
\ft{g}(x):=\int_{-\infty}^{\infty}g(t)e^{-ixt}\,dt, \qquad
x\in\R.
\end{equation}
It is known that $\widehat{g}$ is uniformly
continuous on $\R$, and that $\lim_{x\to\pm\infty}
\widehat{g}(x)=0$.
In general $\widehat{g}$ need not
be integrable, but if it is, and if 
$g\in\CR$ (the space of functions
which are continuous throughout the real line), then
one obtains the following {\it inversion formula\/}:
$$
g(t)=\frac{1}{\Lwoodone}\int_{-\infty}^{\infty}\widehat{g}(x)
e^{ixt}\,dx, \qquad t\in\R.
$$

Suppose now that $g\in\LtwoR$.  
The Fourier transform
of $g$, denoted by $\FT{g}$, is the function
in $\LtwoR$ for which $\lim_{N\to\infty}\|h_N-{\cF}[g]\|_{\LtwoR}=0$,
where
\[
h_N(x):=\int_{-N}^N g(t)e^{ixt}\,dt, \qquad x\in\R.
\]
The integral above is finite for every real number $x$, because
$g$ is square integrable on $\R$, hence locally
integrable on $\R$.  As ${\cF}[g]$ is obtained
({\it a priori\/}) only as an element in $\LtwoR$, 
it is determined only 
almost everywhere.  It is known that
${\cF}$ is a linear isomorphism on $\LtwoR$, and that
the following hold:
\begin{equation}\label{E:0.1}
\|{\cF}[g]\|_{\LtwoR}^2 =\Lwoodone \|g\|_{\LtwoR}^2,
\;g\in\LtwoR; 
\quad
{\cF}[g]=\widehat{g},\;
g\in\LtwoR\cap\LoneR.
\end{equation}
The inversion formula for the Fourier transform of
square-integrable functions takes the following 
form: $\lim_{N\to\infty}\|g-H_N\|_{\LtwoR}=0$, where
\begin{equation}\label{E:0.2}
H_N(t):=\frac{1}{\Lwoodone}\int_{-N}^N{\cF}[g](x)e^{ixt}\,dt, \qquad t\in\R.\\
\end{equation}
If, in addition to being square integrable, ${\cF}[g]$
is also integrable on $\R$, then
the Dominated Convergence Theorem implies
that 
\begin{equation}\label{E:0.3}
\lim_{N\to\infty}H_N(t)=
\frac{1}{\Lwoodone}\int_{-\infty}^\infty{\cF}[g](x)e^{ixt}\,dx
=:H(t), \qquad
t\in\R,
\end{equation}
so $g$ must coincide with $H$
almost everywhere.  As $H$
is continuous on $\R$ (in fact,
$2\pi H(t)=\widehat{{\cF}[g]}(-t)$),
we find that, if $g\!\in\!\LtwoR\cap\CR$ and ${\cF}[g]\!\in\!L_1(\R)$, then
\begin{equation}\label{E:0.4}
g(t)=\frac{1}{\Lwoodone}\int_{-\infty}^\infty{\cF}[g](x)e^{ixt}\,dx, 
\qquad t\in\R.
\end{equation}

The functions we seek to interpolate 
are the so-called {\it bandlimited\/}
or {\it Paley--Wiener functions\/}.  Specifically,
we define
\[
\PW:=\big\{g\in\LtwoR\,:\,{\cF}[g]=0\hbox{\ 
almost everywhere outside }[-\pi,\pi]\big\}.
\]
The first equation in \eqref{E:0.1} leads
to the finding that $\PW$ is a closed
subspace of $\LtwoR$.  Moreover, if
$g\in\PW$, then ${\cF}[g]\in\LtwoR$
and ${\cF}[g]=0$ almost everywhere
outside $[-\pi,\pi]$, so 
${\cF}[g]\in\LoneR$; hence
the inversion formula discussed in the 
foregoing paragraph
asserts that
\begin{equation}\label{E:0.5}
g(t)=\frac{1}{\Lwoodone}\int_{-\infty}^{\infty}
{\cF}[g](x)e^{ixt}\,dx=
\frac{1}{\Lwoodone}\int_{-\pi}^{\pi}
{\cF}[g](x)e^{ixt}\,dx, 
\end{equation}
for almost every real number $t$.  Therefore,
by altering the values of $g$, if need be,
on a Lebesgue-null set, we may assume 
that \eqref{E:0.5} holds for {\sl every\/} real number
$t$, {\sl and this assumption will be in place
throughout the article\/}.
In particular we shall assume that every function
in $\PW$ is continuous throughout $\R$.
Moreover,  
the Bunyakovskii--Cauchy--Schwarz 
Inequality and \eqref{E:0.1} 
combine to show that $g$ is also a bounded function:
\begin{equation}\label{E:0.5a}
|g(t)|\le \frac1{\sqrt{2\pi}}\|\FT{g}\|_{L_2[-\pi,\pi]}=\|g\| _{L_2(\R)}, 
\quad t\in\R.
\end{equation}
 Even though it is
not relevant here, we mention,
at least by way of explaining our terminology,
that the celebrated Paley--Wiener Theorem
proclaims that a function belongs to $\PW$ 
if and only if it can be extended
to the complex plane as an
entire function of exponential type at most $\pi$.

Having discussed the functions which are to 
be interpolated, we now describe
the (canonical) point sets at which these
functions will be interpolated; it is 
customary to refer to these points as
{\it interpolation points\/}, or
{\it sampling points\/}, or {\it data sites\/}.
For the most part, though not always, we shall
be concerned with data sites which
give rise to {\it Riesz-basis sequences\/}.  
Precisely,
following \cite{LM}, we say that
a real sequence $(x_j:j\in\Z)$ is
an Riesz-basis sequence if it satisfies
the following conditions: $x_j<x_{j+1}$ for
every integer $j$, and the sequence of functions
$(e_j(t):=e^{-ix_jt}: j\in\Z,\; t\in\R)$ is
a {\it Riesz basis\/} for $L_2[-\pi,\pi]$.  We
recall that saying that a sequence $(\varphi_j:j\in\Z)$
in a Hilbert space ${\mathcal H}$ is a
Riesz basis for ${\mathcal H}$ means that
{\sl every\/} 
element $h\in {\mathcal H}$
admits a unique representation of the form
\begin{equation}\label{E:0.6}
h=\sum_{j\in\Z}a_j\varphi_j, \qquad \sum_{j\in\Z}|a_j|^2<\infty,
\end{equation}
and that there exists a universal constant $B$ such that
\begin{equation}\label{E:0.6b}
B^{-1}\left(\sum_{j\in\Z}|c_j|^2\right)^{1/2}\le
\left\|\sum_{j\in\Z}c_j\varphi_j\right\|_{{\mathcal H}}\le B
\left(\sum_{j\in\Z}|c_j|^2\right)^{1/2},
\end{equation}
for {\sl every\/} square-summable sequence $(c_j:j\in\Z)$.
Classical examples
of Riesz-basis sequences 
are given in
\cite{LM}, where it is also pointed out
that if $(x_j:j\in\Z)$ is an Riesz-basis sequence, then
there exist positive numbers $q$ and $Q$ such that
\begin{equation}\label{E:0.7}
q\le x_{j+1}-x_j\le Q, \qquad j\in\Z.
\end{equation}

The interpolation process we study here
is one that arises from translating
a fixed Gaussian.  Specifically, let
$\lambda>0$ be fixed, and let
$(x_j:j\in\Z)$ be an Riesz-basis sequence
(a relaxation of this condition will be
considered in the final section of the paper).
We show in the next section that 
given a function $f\in\PW$, there
exists a unique square-summable sequence  
$(a_j:j\in\Z)$ -- depending on
$\lambda$, $f$, and the sampling
points $(x_j)$ --  such that 
the function
\begin{equation}\label{E:0.8}
\Il(f)(x):=\sum_{j\in\Z}a_je^{-\lambda(x-x_j)^2}, \qquad
x\in\R,
\end{equation}
is continuous and square-integrable on 
$\R$, and satisfies the interpolatory conditions
\begin{equation}\label{E:0.9}
\Il(f)(x_k)=f(x_k), \qquad k\in\Z.
\end{equation}
The function $\Il(f)$ is 
called the {\it Gaussian Interpolant\/} to $f$
at the data sites $(x_k:k\in\Z)$.  We 
also prove, again in the upcoming section, that
the map $f\mapsto\Il(f)$ is a bounded linear
operator from $\PW$ to $\LtwoR$.  As
expected, the norm of this 
operator $\Il$ -- which we refer to as the
{\it Gaussian Interpolation Operator\/} --  
is shown to be bounded
by a constant depending on $\lambda$ and the
choice of the Riesz-basis sequence.  However,
this is not sufficient for our subsequent
analysis, in which we intend to vary
the scaling parameter $\lambda$.  So in Section
3 we demonstrate that, if the underlying RRB
sequence is fixed, and if $\lambda\le1$ 
(the upper bound $1$ being purely a matter of convenience), then
the operator norm of $\Il$ can be majorized by
a number which is {\sl independent\/} of 
$\lambda$. Armed with this finding, we proceed to
Section 4, wherein we establish the following
focal convergence result:
\begin{theorem}\label{T:0.1}
Suppose that $(x_j:j\in \Z)$ is a (fixed) Riesz-basis sequence, 
and let $\Il$ be the associated Gaussian Interpolation Operator.
Then for any $f\in\PW$, we have
 $f=\lim_{\lambda\to0^+} \Il(f)$ in $L_2(\R)$ and uniformly on $\R$.
\end{theorem}
\noindent We note that, in the case 
when $x_j=j$, this theorem was proved in \cite{BS}. 

Our proofs in Sections 2--5 rely heavily on the machinery 
and methods developed
in \cite{LM} for cardinal splines; indeed, 
as mentioned earlier, our primary task 
in this paper has been to adapt these
to the study of the Gaussian.  However, most of
these {\sl arguments\/}
do not extend {\it per se\/} to the multidimensional situation, which
occupies our attention in Section 5.  The results presented
in this section are far from complete, and should be viewed
only as partial generalizations of their univariate counterparts.
Nonetheless, it is not without interest to note that 
tackling even this simplified situation requires a combination
of the  {\sl results\/} established in one dimension and
some abstract functional-analytic techniques.  The paper
concludes with Section 6, in which we revisit
univariate interpolation, but consider sampling
points which satisfy a less restrictive
condition than that of giving rise
to a Riesz-basis sequence.  We introduce here
the {\it fundamental functions\/} for
interpolation at such data sites, and prove
that they decay exponentially for 
large argument.  In addition to being
of independent interest, as readers
familiar with spline theory will readily attest,
this result also paves the way
towards a generalization of some 
of the main results of
Section 2.

We have attempted to make this article
as self contained as possible, and we 
request the indulgence of those readers
who may find an abundance of detail
between these pages.

\section{Notations and basic facts}\label{S:1}

In this section we shall reintroduce the interpolation problem which
concerns us, define the corresponding interpolant and interpolation
operator, and establish
some of their basic properties.  We shall uncover these in a
series of propositions, which begins with this simple observation.

\begin{proposition}\label{P:1.0}  If $\alpha$ is a positive
number, then
\begin{equation*}
\sum_{l\in\Z\setminus\{0\}}e^{-\alpha(2|l|-1)^2}\le 
\frac{2e^{-\alpha}}{ 1-e^{-\alpha}}=:\kappa(\alpha).
\end{equation*}
\end{proposition}

\begin{proof}  
$$
\sum_{l\in\Z\setminus\{0\}}e^{-\alpha(2|l|-1)^2}\le
2\sum_{l=1}^\infty e^{-\alpha l}=
\frac{2e^{-\alpha}}{ 1-e^{-\alpha}}.
$$
\end{proof}

In what follows we shall use the following notation: given a
positive number $\lambda$, the Gaussian function with
scaling parameter $\lambda$ is defined by
\[
\gl(x):=e^{-\lambda x^2}, \qquad x\in\R.
\]
We recall the well-known fact (see, for example, \cite[p. 43]{Go})
that 
\begin{equation}\label{E:1.0}
{\cF}[\gl](u)={\widehat{\gl}}(u)=\sqrt{\frac{\pi}{\lambda}}\,e^{-u^2/(4\lambda)}, \qquad u\in\R.
\end{equation}

We now record two results from the literature; both will
be of use in this section and also in Section 5.

\begin{proposition}\label{P:NSW} {\rm cf. \cite[Lemma 2.1]{NSW}}\\
Let $\lambda$ and $q$ be fixed positive numbers, and let
$\|\cdot\|_2$ denote the Euclidean norm in 
$\R^d$.  There exists a number $\nu$, depending only
on $d$, $\lam$, and $q$, such that
the following holds: if 
$(x_j)$
is any sequence in $\R^d$ 
with $\|x_j-x_k\|_2\ge q$ for
$ j\neq k$, and $x$ is 
any point in $\R^d$, then
$
\sum_j\gl(\|x-x_j\|_2)\le\nu.
$  
\end{proposition}

This next result is an important finding in the theory
of radial-basis functions.

\begin{theorem}\label{T:NW} {\rm cf. \cite[Theorem 2.3]{NW}}\\
Let $\lambda$ and $q$ be fixed positive numbers, and let
$\|\cdot\|_2$ denote the Euclidean norm in 
$\R^d$.  There exists a number $\theta$, depending only
on $d$, $\lam$, and $q$, such that
the following holds: if 
$(x_j)$
is any sequence in $\R^d$ 
with $\|x_j-x_k\|_2\ge q$ for
$j\neq k$, then
$
\sum_{j,k}\xi_j{\overline \xi}_k\gl(\|x_j-x_k\|_2)\ge\theta
\sum_j|\xi_j|^2,
$
for every sequence of complex numbers
$(\xi_j)$. 
\end{theorem}


\begin{proposition}\label{P:1.1}
Suppose that $(x_j:j\in\Z)$ is a 
sequence
of real numbers satisfying the following condition: there exists a positive
number $q$ such that $x_{j+1}-x_j\ge q$ for every integer $j$.  Let
$\lambda>0$ be fixed, and let
$(a_j:j\in\Z)$ be a bounded sequence of complex numbers.  Then the function
$\R\ni x\mapsto\sum_{j\in\Z}a_j\gl(x-x_j)$ is continuous and bounded throughout the 
real line.
\end{proposition}

\begin{proof} Proposition \ref{P:NSW} 
demonstrates that the series in question is uniformly
convergent throughout $\R$, and that the sum
is a bounded function of $x$.  The apparent
continuity of each summand and uniform convergence
imply the continuity of the limit function.
\end{proof}

\begin{remark} \label{R:invertibility}
Suppose that $\lam$ is a fixed
positive number, and let
$(x_j:j\in\Z)$ be a sequence satisfying
the conditions of Proposition \ref{P:1.1}.
The latter conditions imply that
$|x_j-x_k|\ge|j-k|q$ for every pair
of integers $j$ and $k$, so the entries
of the bi-infinite matrix 
$(g_\lam(x_k-x_j))_{k,j\in\Z}$ decay
exponentially away from its main diagonal.
So the matrix is realizable
as the sum of a uniformly
convergent series of diagonal matrices.
Hence it acts as a bounded operator
on every $\ell_p(\Z)$, $1\le p\le\infty$.
Moreover, as the matrix
is also symmetric, 
Theorem \ref{T:NW} ensures that
it is boundedly invertible on $\ell_2(\Z)$.
In particular,
given a 
square-summable sequence $(d_k:k\in\Z)$, there
exists a unique square-summable sequence
$(a(j,\lambda):j\in\Z)$ such that
\[
\sum_{j\in\Z}a(j,\lambda)\gl(x_k-x_j)=d_k, \qquad k\in\Z.
\]
\end{remark}


Suppose now that $(x_j:j\in\Z)$ is a Riesz-basis sequence.
Thus, given $h\in\LtwoT$, 
there exists a square-summable sequence
$(a_j:j\in\Z)$ such that $h(t)\!=\!\sum_{j\in\Z}a_je^{-ix_jt}$
for almost every $t\in[-\pi,\pi]$.  We wish to 
extend this function to $\R$ as follows:

\begin{proposition}\label{P:1.3}
Let $(x_j:j\in\Z)$, $h$, and $(a_j:j\in\Z)$
be as above.  The function $H(u):=\sum_{j\in\Z}
a_je^{-ix_ju}$ is locally square integrable
on $\R$; in particular it is well defined
for almost every real number $u$.
\end{proposition}

\begin{proof} Recall from the introductory
section that there exists a constant $B$ 
such that
\begin{equation}\label{E:1.1}
B^{-2}\sum_{j\in\Z}|c_j|^2\le
\int_{-\pi}^{\pi}|\sum_{j\in\Z}c_j e^{-ix_jt}|^2\,dt
\le B^2
\sum_{j\in\Z}|c_j|^2
\end{equation}
for every square-summable sequence $(c_j:j\in\Z)$.  Let
$u\in[(2l-1)\pi,(2l+1)\pi]$ for some 
integer $l$, and define, for 
every positive integer $N$, the (continuous) function 
$H_N(u):=\sum_{j=-N}^N a_je^{-ix_ju}$.  If $N>M$ are positive
integers, then
\begin{align}\label{E:1.2}
\int_{(2l-1)\pi}^{(2l+1)\pi}|H_N-H_M|^2&=\int_{(2l-1)\pi}^{(2l+1)\pi}
\left|\sum_{|j|=M+1}^N a_je^{-ix_j u}\right|^2\,du \\
\notag &=\int_{-\pi}^{\pi}\left|\sum_{|j|=M+1}^N a_je^{-2\pi ilx_j}e^{-ix_jv}\right|^2\,dv\\
\notag &\le B^2\sum_{|j|=M+1}^N\left|a_j e^{-2\pi ilx_j}\right|^2=
B^2\sum_{|j|=M+1}^N|a_j|^2,
\end{align}
the final inequality coming from \eqref{E:1.1}.  Ergo, the 
square summability of the sequence $(a_j:j\in\Z)$ shows
that $(H_N: N\in\N)$ is a Cauchy sequence, hence
convergent, in $L_2[(2l-1)\pi,(2l+1)\pi]$.  This
proves the desired result.  
\end{proof}
For future reference, we note
that the argument leading up to 
\eqref{E:1.2} also provides the estimate
\begin{equation}\label{E:1.3}
\int_{(2l-1)\pi}^{(2l+1)\pi}|H|^2=\lim_{N\to\infty}
\int_{(2l-1)\pi}^{(2l+1)\pi}|H_N|^2\le
B^2\sum_{j\in\Z}|a_j|^2,\qquad l\in\Z,
\end{equation}
whence the Bunyakovskii--Cauchy--Schwarz Inequality
implies that
\begin{equation}\label{E:1.4}
\int_{(2l-1)\pi}^{(2l+1)\pi}|H|\le \sqrt{\Lwoodone}\,B
\left(\sum_{j\in\Z}|a_j|^2\right)^{1/2},\qquad l\in\Z.
\end{equation}

The next result is the first of the two main
offerings of the current section.

\begin{theorem}\label{T:1.4}
Suppose that $\lambda$ is a fixed positive number, and let
$(x_j:j\in\Z)$ be a Riesz-basis sequence.  Assume that
${\overline a}:=(a_j:j\in\Z)$ is a square-summable sequence.  The following
hold:
\item{(i)} The function 
\[
s({\overline a},x)=s(x):=\sum_{j\in\Z}a_j\gl(x-x_j), \qquad x\in\R,
\]
belongs to $C(\R)\cap\LtwoR$.
\item{(ii)} The function 
\[
{\tilde s}({\overline a},u)={\tilde s}(u):=e^{-u^2/(4\lambda)}\sum_{j\in\Z}a_je^{-ix_ju}
\]
is well defined for almost every real number $u$, and 
${\tilde s}\in\LtwoR\cap\LoneR$.
\item{(iii)}  
\begin{equation*}
{\cF}[s]=\sqrt{\frac{\pi}{\lambda}}\;{\tilde s}.
\end{equation*}
\item{(iv)} The map ${\overline a}:=(a_j:j\in\Z)\mapsto
s({\overline a},x):=
\sum_{j\in\Z}a_j\gl(x-x_j)$, $x\in\R$, is a bounded
linear transformation from $\ell_2(\Z)$ into $\LtwoR$.
\end{theorem}

\begin{proof} (i) The continuity of $s$ on $\R$ follows at once
from the first inequality in \eqref{E:0.7} and
Proposition \ref{P:1.1}. Define $s_N(x):=
\sum_{j=-N}^N a_j\gl(x-x_j)$, $x\in\R$, $N\in\N$.  
As
$s_N\in\LoneR\cap\LtwoR$ for every positive integer $N$, equation
\eqref{E:0.1}, \eqref{E:1.0}, and a standard
Fourier-transform calculation provide the 
following relations: if $N>M$ are positive integers, then
\begin{align}\label{E:1.5}
2\pi\|s_N-s_M\|_{\LtwoR}^2&=\|{\widehat{s_N}}-{\widehat{s_M}}\|_{\LtwoR}^2\\
\notag &=\frac{\pi}{\lambda}
\int_{-\infty}^{\infty}e^{-u^2/(2\lambda)}\left|\sum_{|j|=M+1}^N
a_je^{-ix_ju}\right|^2\,du\\
\notag &=\frac{\pi}{\lambda}\sum_{l\in\Z}\int_{(2l-1)\pi}^{(2l+1)\pi}e^{-u^2/(2\lambda)}
\left|\sum_{|j|=M+1}^Na_je^{-ix_ju}\right|^2\,du.
\end{align}
Let $H_{M,N}(u):=\sum_{|j|=M+1}^N a_j e^{-ix_ju}$, $u\in\R$.
Using the estimates $e^{-u^2/(2\lambda)}\le1$ for $u\in[-\pi,\pi]$, and
$e^{-u^2/(2\lambda)}\le e^{-(2|l|-1)^2\pi^2/(2\lambda)}$ 
for $(2l-1)\pi\le u\le(2l+1)\pi$, $l\in\Z\setminus\{0\}$, we find
from \eqref{E:1.5} that  
\begin{align}\label{E:1.6}
2\pi\|s_N\!-s_M\|_{\LtwoR}^2
  \le\!\frac{\pi}{\lambda}
&\Big[\left\|H_{M,N}\right\|_{\LtwoT}^2\\
&+\!\!\!\!
\sum_{l\in\Z\setminus\{0\}}\!e^{-(2|l|-1)^2\pi^2/(2\lambda)}
\left\|H_{M,N}\right\|_{L_2[-(2l-1)\pi,(2l+1)\pi]}^2\Big].\notag
\end{align}
Now let $B$ be the constant satisfying \eqref{E:1.1} 
and use \eqref{E:1.2} to estimate
each of the integrals on the right side of \eqref{E:1.6}.
This yields the relation
\begin{equation*}
2\pi\|s_N-s_M\|_{\LtwoR}^2\le\left(\frac{\pi}{\lambda}\right)
B^2\left(\sum_{|j|=M+1}^N|a_j|^2\right)
\left[1+\sum_{l\in\Z\setminus\{0\}}e^{-(2|l|-1)^2\pi^2/(2\lambda)}
\right],
\end{equation*}
whence Proposition \ref{P:1.0} leads to the estimate
\begin{equation*}
2\pi\|s_N-s_M\|_{\LtwoR}^2\le\left(\frac{B^2\pi}{\lambda}\right)
\left(\sum_{|j|=M+1}^N|a_j|^2\right)\left[1+\kappa(\pi^2/(2\lambda))\right].
\end{equation*}
As $(a_j:j\in\Z)$ is square summable, we find that 
$(s_N:N\in\N)$ is a Cauchy sequence in $\LtwoR$, and hence
that $s\in\LtwoR$ as promised.

(ii) Let $H(u):=\sum_{j\in\Z}a_je^{-ix_ju}$.  As observed
in Proposition \ref{P:1.3}, $H$ is defined almost
everywhere on $\R$, so the same is true of $\tilde s$
as well.  Now the argument in (i), combined with
\eqref{E:1.3}, shows that
\begin{equation}\label{E:1.7}
\|{\tilde s}\|_{\LtwoR}^2\le B^2\left(
\sum_{j\in\Z}|a_j|^2\right)\left[
1+\kappa(\pi^2/(2\lambda))\right],
\end{equation}
whilst a slight, but obvious, variation on the theme, coupled with
\eqref{E:1.4}, demonstrates that
\begin{equation}\label{E:1.8}
\|{\tilde s}\|_{\LoneR}\le \sqrt{\Lwoodone}\,B
\left(
\sum_{j\in\Z}|a_j|^2\right)^{1/2}\left[
1+\kappa(\pi^2/(2\lambda))\right],
\end{equation}
and this completes the proof.

(iii) Let $(s_N:N\in\N)$ be the sequence defined in the proof of (i).
As each $s_N\in\LoneR\cap\LtwoR$ and $\lim_{N\to\infty}
\|s_N-s\|_{\LtwoR}=0$, it suffices to show, thanks to 
\eqref{E:0.1}, that 
$\lim_{N\to\infty}\|{\widehat s_N}-\sqrt{(\pi/\lambda)}\;{\tilde s}\|_{\LtwoR}=0$.
Calculations similar to the one carried out in (i) show that 
\begin{align*}
\|{\widehat s_N}-\sqrt{(\pi/\lambda)}\;{\tilde s}\|_{\LtwoR}^2&=
\frac{\pi}{\lambda}\int_{-\infty}^{\infty}e^{-u^2/(2\lambda)}
\left|\sum_{|j|=N+1}^\infty a_je^{-ix_ju}\right|^2\,du\\
&\le
\frac{B^2\pi}{\lambda}\left(\sum_{|j|=N+1}^\infty|a_j|^2\right)^2
\left[
1+\kappa(\pi^2/(2\lambda))\right],
\end{align*}
and the last term approaches zero as $N$ tends to 
infinity,
because the
sequence $(a_j:j\in\Z)$ is square summable.

(iv) The linearity of the map is evident, and
that it takes $\ell_2(\Z)$ into $\LtwoR$
is the content of part (i).  Now
\eqref{E:0.1}, part (iii) above, and
\eqref{E:1.7} combine to yield the relations
\begin{equation*}
\Lwoodone\left\|s({\overline a},\cdot)\right\|_{\LtwoR}^2=
\left\|\sqrt{\frac{\pi}{\lam}}\;{\tilde s}({\overline a},\cdot)\right\|_{\LtwoR}^2\le
\frac{B^2\pi}{\lam}\left[
1+\kappa(\pi^2/(2\lambda))\right]\left\|{\overline a}\right\|_{\ell_2(\Z)}^2.
\end{equation*}
\end{proof}

This next result points to 
a useful interplay between Riesz-basis sequences
and bandlimited functions (see, for example,
\cite[pp.~29-32]{Yo}). 
It
serves as 
a prelude to the second main theorem
of this section.

\begin{proposition}\label{P:1.5} 
Suppose that $(x_j:j\in\Z)$ is 
a Riesz-basis sequence, and that $f\in\PW$.  
Then the (sampled) sequence  
$(f(x_j):j\in\Z)$ is square summable.  Moreover,
there is an absolute constant $C$ -- depending
only on $(x_j)$, but not on $f$ -- such that
$\sum_{j\in\Z}|f(x_j)|^2\le C^2\|f\|_{\LtwoR}^2$.
\end{proposition}

\begin{proof}
Let 
\begin{equation}\label{E:1.9}
\langle h_1,h_2\rangle:=\int_{-\pi}^\pi h_1{\overline h_2}, \qquad
h_1,h_2\in\LtwoT,
\end{equation}
denote the standard inner product in $\LtwoT$.
Let $e_j(t):=e^{-ix_jt}$, $j\in\Z$, $t\in[-\pi,\pi]$, so that
\eqref{E:0.5} implies the identities
\begin{equation}\label{E:1.10}
\Lwoodone f(x_j)=\langle{\mathcal F}[f],e_j\rangle, \qquad j\in\Z.
\end{equation}
 Letting $(\tilde e_j:j\in\Z)$ be the co-ordinate functionals
 of  $( e_j:j\in\Z)$ (which means
 that $h=\sum_j \langle h,\tilde e_j\rangle e_j$ for any $h\in L_2[\-\pi,\pi]$),
 it follows that $(\tilde e_j:j\in\Z)$ is also a Riesz basis whose co-ordinate functionals
 are $( e_j:j\in\Z)$. Thus,
\begin{equation}\label{E:1.11}
g=\sum_{j\in\Z}\langle g,e_j\rangle{\tilde e}_j,\quad g\in\LtwoT.
\end{equation}
  So \eqref{E:0.6b} provides a universal
constant ${\tilde B}$ such that 
\begin{equation*}
\sum_{j\in\Z}|c_j|^2\le {\tilde B}^2\left\|\sum_{j\in\Z}c_j{\tilde e}_j\right\|_{\LtwoT}^2
\end{equation*}
for every square summable sequence $(c_j:j\in\Z)$.  Hence 
\eqref{E:1.10} and \eqref{E:1.11} imply that
\begin{equation*}
4\pi^2\sum_{j\in\Z}|f(x_j)|^2\le{\tilde B}^2\left\|{\mathcal F}[f]\right\|_{\LtwoT}^2  
=\Lwoodone{\tilde B}^2\|f\|_{\LtwoR}^2,
\end{equation*}
the final equation stemming from \eqref{E:0.1}, and the fact that
${\mathcal F}[f]=0$ almost everywhere in $\R\setminus[-\pi,\pi]$.
\end{proof}

We now state the second of the two main results in
this section.  Most of our work has already
been accomplished; what
remains is to recast the findings in the context
of our interpolation problem.

\begin{theorem}\label{T:1.6}
Let $\lambda$ be a fixed positive number, and  
let 
$(x_j:j\in\Z)$ be a Riesz-basis sequence.
The following hold:
\item{(i)} Given $f\in\PW$, there exists
a unique square-summable sequence $(a(j,\lambda):j\in\Z)$
such that 
\begin{equation*}
\sum_{j\in\Z}a(j,\lambda)g_\lambda(x_k-x_j)=f(x_k), \qquad k\in\Z.
\end{equation*}
\item{(ii)} Let $f$ and $(a(j,\lambda):j\in\Z)$ be as in (i).
The Gaussian interpolant to $f$ at the points $(x_j:j\in\Z)$,
to wit,
\begin{equation*}
\Il(f)(x)=\sum_{j\in\Z}a(j,\lambda)g_\lambda(x-x_j), \qquad x\in\R,
\end{equation*}
belongs to $C(\R)\cap\LtwoR$.
\item{(iii)} Let $f$ and $\Il(f)$ be as above. 
The Fourier transform of $\Il(f)$
is given by 
\begin{equation*}
{\mathcal F}[\Il(f)](u)=\sqrt{\frac{\pi}{\lambda}}\,e^{-u^2/(4\lambda)}
\sum_{j\in\Z}a(j,\lambda)e^{-ix_ju}=:
\sqrt{\frac{\pi}{\lambda}}\,e^{-u^2/(4\lambda)}\Psi_\lam(u)
\end{equation*}
for almost every real number $u$.  Moreover, 
${\mathcal F}[\Il(f)]\in\LtwoR\cap\LoneR$.
\item{(iv)} If $f$ and $\Il(f)$ are as above, then
\begin{equation*}
\Il(f)(x)=\frac{1}{\Lwoodone}\int_{-\infty}^{\infty}{\mathcal F}[\Il(f)](u)e^{ixu}\,du,
\qquad x\in\R.
\end{equation*}
In particular, 
$\Il(f)\in C_0(\R):=\lbrace g\in\CR:\lim_{|x|\to\infty}g(x)=0\rbrace$.
\item{(v)} The Gaussian interpolation operator $\Il$ is
a bounded linear operator from $\PW$
to $\LtwoR$.  That is, 
the map $\PW\ni f\mapsto\Il(f)$ is linear, and 
there exists a positive constant $D$,
depending only on $\lambda$ and $(x_j:j\in\Z)$, such that
\begin{equation*}
\left\|\Il(f)\right\|_{\LtwoR}\le D\left\|f\right\|_{\LtwoR}
\\
\end{equation*}
for every $f\in\PW$.  
\end{theorem}

\begin{proof}
Assertion (i) follows from the first inequality
in \eqref{E:0.7}, 
Remark \ref{R:invertibility}, and Proposition
\ref{P:1.5}. Assertion (ii) obtains from 
part (i) of Theorem \ref{T:1.4}, whilst assertion (iii)
is a consequence of parts (ii) and (iii) of 
Theorem \ref{T:1.4}. Assertion (iv)
follows from the fact that 
$\Il(f)$ is continuous throughout $\R$, 
that ${\mathcal F}[\Il(f)]\in\LtwoR\cap\LoneR$, and 
equation \eqref{E:0.4}.  
Moreover, this representation and equation \eqref{E:wrongdef}
show that 
$2\pi\Il(f)(x)=\widehat{{\FT{\Il(f)}}}(-x)$ for every
real number $x$.  So the Riemann--Lebesgue
Lemma \cite[Theorem 4A]{Go} ensures that
$\Il\in C_0(\R)$.
As to (v), 
let
$T$ be the matrix $(g_\lam(x_k-x_j))_{j,k\in\Z}$,
let $f\in\PW$, and let ${\overline d}:=(f(x_j):j\in\Z)$.
Then $a(j,\lambda)$ is the $j$-th component 
of the vector $T^{-1}{\overline d}$, and this 
demonstrates that $\Il$ is linear.  Now 
part (iv) of Theorem \ref{T:1.4} asserts
that
\begin{equation}\label{E:1.12}
\left\|\Il(f)\right\|_{\LtwoR}=
O\left(\left\|T^{-1}{\overline d}\right\|_{\ell_2(\Z)}\right)
=O\left(\left\|{\overline d}\right\|_{\ell_2(\Z)}\right),
\end{equation}
where the Big-O constant depends only on $\lambda$ and 
the Riesz-basis sequence $(x_j:j\in\Z)$.
Furthermore, Proposition \ref{P:1.5} reveals that
\begin{equation}\label{E:1.13}
\left\|{\overline d}\right\|_{\ell_2(\Z)}=
O\left(\left\|f\right\|_{\LtwoR}\right),
\end{equation}
with the Big-O constant here depending only on
$(x_j:j\in\Z)$.
Combining \eqref{E:1.12} with \eqref{E:1.13}
finishes the proof.
\end{proof}

\section{Uniform boundedness of the interpolation operators}\label{S:2}

In the final result of the previous section, it was shown
that, for a fixed 
scaling parameter $\lam$,  and a 
fixed Riesz-basis sequence $(x_j:j\in\Z)$,  
the 
associated interpolation operator $\Il$ is a continuous
linear map from $\PW$ into $\LtwoR$.  As expected,
the norm of this operator was shown to be bounded
by a number which depends on both the scaling parameter and 
the choice of the Riesz-basis sequence.  The goal in the current
section is to demonstrate that, 
if 
the scaling parameter is
bounded above by a fixed number (taken here
to be 1 for convenience), then the norm
of $\Il$ can
be bounded by a number which depends only
on $(x_j:j\in\Z)$.  The proofs
in this section (as well as in the next) are 
patterned after \cite{LM}.  

Let $(x_j:j\in\Z)$ be a Riesz-basis sequence, and let 
$B$ be the associated constant satisfying 
the inequalities in \eqref{E:1.1}.  Given
$h\in\LtwoT$, there is a square-summable sequence $(a_j:j\in\Z)$
such that 
$h(t)=\sum_{j\in\Z}a_je^{-ix_jt}$ for almost every
$t\in[-\pi,\pi]$.  Let $H$ denote the extension
of $h$ to almost all of $\R$, as considered
in Proposition \ref{P:1.3}.  Given an integer $l$,
we define the following linear map $A_l$
on $\LtwoT$:
\begin{equation}\label{E:2.0}
A_l(h)(t):=H(t+2\pi l)=\sum_{j\in\Z}a_je^{-ix_j(t+2\pi l)}
\end{equation}
for almost every $t\in[-\pi,\pi]$.    
We see from \eqref{E:1.3} and \eqref{E:1.1} that
\begin{equation}\label{E:2.1}
\left\|A_lh\right\|_{\LtwoT}^2=
\left\|H\right\|_{L_2[(2l-1)\pi,(2l+1)\pi]}^2\le
B^2\sum_{j\in\Z}|a_j|^2\le
B^4\left\|h\right\|_{\LtwoT}^2.
\end{equation}
Thus every $A_l$ is a bounded operator from
$\LtwoT$ into itself; moreover,
the associated operator norms of these operators
are uniformly bounded:
\begin{equation}\label{E:2.2}
\left\|A_l\right\|\le B^2.
\end{equation}

In what follows, we shall assume that 
$(x_j:j\in\Z)$ is a (fixed) Riesz-basis sequence,
and let $e_j(t):=e^{-ix_j t}$, $t\in\R$,
$j\in\Z$.  We also denote by
$\langle\cdot,\cdot\rangle$ the standard
inner product in $\LtwoT$, as defined via
\eqref{E:1.9}.  Our first main
task now is to 
exploit the presence of the Riesz basis
$(e_j:j\in\Z)$ in $\LtwoT$ to 
find an effective
representation for the Fourier transform
of the Gaussian interpolant to a 
given bandlimited function, on the 
interval $[-\pi,\pi]$.
We begin
with a pair of preliminary observations:

\begin{lemma}\label{L:2.0}
Let $(x_j)$ and $(e_j)$ be as above,
and let $f$ be a given function
in $\PW$.
If $\phi\in\LtwoT$ satisfies the 
conditions
\begin{equation}\label{E:2.3}
\Lwoodone f(x_k)=\int_{-\pi}^{\pi}\phi(t)e^{-ix_kt}\,dt, \qquad
k\in\Z,
\end{equation}
then ${\mathcal F}[f]$ agrees with $\phi$ in $\LtwoT$.
\end{lemma}

\begin{proof} Equations \eqref{E:1.10} and 
\eqref{E:2.3} reveal that 
$\langle{\mathcal F}[f],e_k\rangle=\langle \phi,e_k \rangle$
for every integer $k$, and the required result follows
from \eqref{E:1.11}.  
\end{proof}

\begin{lemma}\label{L:2.1}
Let
$(x_j)$ and $(e_j)$ be as above, and let 
$B$ be the constant satisfying \eqref{E:1.1}.
Let $h\in\LtwoT$, and let $\alpha>0$. Define
$$\phi_l = A_l^{{}^*}\left(
e^{-\alpha(\cdot+\Lwoodone l)^2}A_l(h)\right),\quad k\in\Z.$$
Then 
\begin{equation}\label{E:2.4}
\|\phi_0 \|_{\LtwoT}\le\|h\|_{\LtwoT}
\hbox{ and }
\sum_{l\in{\Z\setminus\{0\}}}\left\|
 \phi_l \right\|_{\LtwoT}\le\left\|h\right\|_{\LtwoT}
B^4\kappa(\pi^2\alpha),
\end{equation}
where $\kappa$ is the familiar function from 
Proposition \ref{P:1.0}.  In particular, the series
$\sum_{l\in\Z}\phi_l$ converges 
in $\LtwoT$.
\end{lemma}

\begin{proof}
We note that 
\begin{equation}\label{E:2.6}
\left\|\phi_0\right\|_{\LtwoT}=\left\|e^{-\alpha(\cdot)^2}h\right\|_{\LtwoT}
\le \|h\|_{\LtwoT},
\end{equation}
whereas
\eqref{E:2.2} gives rise to the following estimates
for every $l\in\Z\setminus\{0\}$:
\begin{align}\label{E:2.7}
\left\|\phi_l\right\|_{\LtwoT}&\le
B^2\left\|
e^{-\alpha(\cdot+\Lwoodone l)^2}A_l(h)\right\|_{\LtwoT}\\
\notag &\le
B^2e^{-\pi^2 \alpha(2|l|-1)^2}\left\|A_lh\right\|_{\LtwoT}\\
\notag &\le
B^4e^{-\pi^2 \alpha(2|l|-1)^2}\|h\|_{\LtwoT}.
\end{align}
This completes the proof.
\end{proof}

We are now ready for the first main result of this
section.

\begin{theorem}\label{T:2.2}
Let $\lam>0$ be fixed, and let $f\in\PW$.  
Let $\psi_\lam$ denote the restriction,
to the interval $[-\pi,\pi]$, of the 
function $\Psi_\lam$ given in 
part (iii) of Theorem \ref{T:1.6}.
Then
\begin{equation}\label{E:2.9}
{\mathcal F}[f]={\mathcal F}[\Il(f)]+
\sqrt{\frac{\pi}{\lam}}\summ
\Alstar\left(
e^{-(\cdot+\Lwoodone l)^2/(4\lam)}A_l\left(\psi_\lam\right)
\right) \text{ on $[-\pi,\pi]$.}
\end{equation}

\end{theorem}

\begin{proof}
Let $\phi$ denote the function on the right side
of \eqref{E:2.9}.  In view of Lemma \ref{L:2.0},
it suffices to show that 
\begin{equation}\label{E:2.10}
\Lwoodone f(x_k)=\langle\phi,e_k\rangle,\qquad k\in\Z.
\end{equation}
Now Theorem \ref{T:1.6} implies the relations
\begin{equation}\label{E:2.11}
\Lwoodone f(x_k)=\Lwoodone\Il(f)(x_k)=
\sqrt{\frac{\pi}{\lam}}
\int_{-\infty}^{\infty}e^{-u^2/(4\lam)}\Psi_\lam(u)e^{ix_ku}\,du,
\qquad k\in\Z,
\end{equation}
whilst
\begin{align}\label{E:2.12}
\int_{-\infty}^{\infty\!\!}e^{-u^2/(4\lam)}\Psi_\lam(u)e^{ix_ku}du
&=\sum_{l\in\Z}\!\!\int_{(2l-1)\pi}^{(2l+1)\pi}\!\!e^{-u^2/(4\lam)}\Psi_\lam(u)e^{ix_ku}du\\
\notag &=\sum_{l\in\Z}\!\!\int_{-\pi}^\pi \!\!
e^{-(t+\Lwoodone l)^2/(4\lam)}\Psi_\lam(t+2\pi l)e^{ix_k(t+2\pi l)}dt\\
\notag &=\sum_{l\in\Z}\!\!\int_{-\pi}^\pi 
\!\!e^{-(t+\Lwoodone l)^2/(4\lam)}A_l(\psi_\lam)(t)
{\overline{A_l(e_k)(t)}}dt\\
\notag &=\sum_{l\in\Z}\left\langle
e^{-(\cdot+\Lwoodone l)^2/(4\lam)}A_l(\psi_{\lam}),
A_l(e_k)
\right\rangle \\
\notag &=\sum_{l\in\Z}\left\langle
\Alstar\left(e^{-(\cdot+\Lwoodone l)^2/(4\lam)}
A_l(\psi_{\lam})\right),e_k
\right\rangle \\
\notag &=\left\langle
\sum_{l\in\Z}\Alstar\left(e^{-(\cdot+\Lwoodone l)^2/(4\lam)}
A_l(\psi_{\lam})\right),e_k
\right\rangle, 
\end{align}
the final step being justified by 
Lemma \ref{L:2.1}.  Noting that 
\begin{equation*}
\sum_{l\in\Z}\Alstar\left(e^{-(\cdot+\Lwoodone l)^2/(4\lam)}
A_l(\psi_{\lam})\right)=
e^{-(\cdot)^2/(4\lam)}\psi_\lam +\summ
\Alstar\left(e^{-(\cdot+\Lwoodone l)^2/(4\lam)}
A_l(\psi_{\lam})\right),
\end{equation*}
we find that \eqref{E:2.12}, \eqref{E:2.11}, and 
part (iii) of Theorem \ref{T:1.6} yield
\eqref{E:2.10}, and with it the proof.
\end{proof}

Combining \eqref{E:2.9}
and \eqref{E:2.4} leads directly to the following:

\begin{corollary}\label{C:2.3}
Suppose that $\kappa$, $\lam$, $f$, $\psi_\lam$, 
and $B$ are as before.
Then
\begin{equation*}
\left\|{\mathcal F}[\Il(f)]\right\|_{\LtwoT}\le
\left\|{\mathcal F}[f]\right\|_{\LtwoT}+
\sqrt{\frac{\pi}{\lam}}\,B^4\kappa(\pi^2/(4\lam))
\left\|\psi_\lam\right\|_{\LtwoT}.
\end{equation*}
\end{corollary}

The preceding corollary shows that,
if $f$ is bandlimited, then 
the energy of the Fourier transform
of (its Gaussian interpolant) $\Il(f)$ -- 
on the interval
$[-\pi,\pi]$ -- is controlled
by that of the Fourier transform
of $f$ on that interval, plus
another term which involves the energy
of $\psi_\lam$ on the interval.  Our next
task is to show that this second term
can also be bounded effectively
via the energy of $\cF[f]$ on 
$[-\pi,\pi]$.  Before proceeding with
this, however, we pause to 
consider formally,
an operator which
has already made its debut, albeit indirectly, 
in Theorem \ref{T:2.2}.
This operator
will also play a role later in this section,
and a larger one in the next.
 
Given a positive number $\alpha$, we define the 
operator $\Top{\alpha}$ on $\LtwoT$ as follows:
\begin{equation}\label{E:2.13}
\Top{\alpha}(h):=e^{\pi^2\alpha}\summ\Alstar\left(
e^{-\alpha(\cdot+2\pi l)^2}A_l(h)
\right), \qquad h\in\LtwoT.
\end{equation}
That this operator is well defined is guaranteed by
Lemma \ref{L:2.1}, whilst 
its linearity is plain. The following 
properties of $\Top{\alpha}$
are easy to verify.

\begin{proposition}\label{P:2.4}
The operator $\Top{\alpha}$ is self adjoint, positive,
and its norm is no larger than 
$e^{\pi^2\alpha}B^4\kappa(\pi^2 \alpha)$, where
$\kappa$ is the function defined through
Proposition \ref{P:1.0}, and $B$ is the familiar
constant associated to the given Riesz-basis sequence $(x_j)$.
\end{proposition}

We now return to the task of carrying 
forward the estimate in Corollary \ref{C:2.3}.
The first order
of business is to attend to 
$\left\|\psi_\lam\right\|_{\LtwoT}$:

\begin{proposition}\label{P:2.5}
The following holds:
\begin{equation*}
\left\|\psi_\lam\right\|_{\LtwoT}\le
\sqrt{\frac{\lam}{\pi}}\;e^{\pi^2/(4\lam)}\left\|
\FT{f}\right\|_{\LtwoT}.
\end{equation*}
\end{proposition}

\begin{proof}
Equation \eqref{E:2.9} asserts that
\begin{equation}\label{E:2.14}
\FT{f}=\FT{\Il(f)}+e^{-\pi^2/(4\lam)}\sqrt{\frac{\pi}{\lam}}
\,\Top{1/(4\lam)}(\psi_\lam).
\end{equation}
As 
\begin{equation}\label{E:2.15}
\left\langle\FT{\Il(f)},\psi_\lam\right\rangle=
\sqrt{\frac{\pi}{\lam}}\lint{-\pi}{\pi}e^{-u^2/(4\lam)}\left|
\psi_\lam(u)\right|^2\,du\ge0,
\end{equation}
and $\Top{1/(4\lam)}$ is a positive operator,
we find from \eqref{E:2.14} that 
$\langle\FT{f},\psi_\lam\rangle$ is 
nonnegative, and also that 
$\langle\FT{f},\psi_\lam\rangle\ge\langle\FT{\Il(f)},\psi_\lam\rangle$.
Hence the Bunyakovskii--Cauchy--Schwarz inequality
and \eqref{E:2.15} lead to the relations
\begin{align*}
\left\|\FT{f}\right\|_{\LtwoT}\lnorm\psi_\lam\rnorm_{\LtwoT}&\ge
\sqrt{\frac{\pi}{\lam}}\lint{-\pi}{\pi}e^{-u^2/(4\lam)}\left|
\psi_\lam(u)\right|^2\,du\\
&\ge
\sqrt{\frac{\pi}{\lam}}e^{-\pi^2/(4\lam)}
\lnorm\psi_\lam\rnorm_{\LtwoT}^2,
\end{align*}
and the required result follows directly.
\end{proof}

The upcoming corollary is obtained via a combination of
Corollary \ref{C:2.3}, Proposition \ref{P:2.5}, Equation \eqref{E:0.1}
and the  fact that 
$e^{\pi^2/4\lam}\kappa(\pi^2/(4\lam))=
(1-e^{-\pi^2/4\lam})^{-1}\le 2$, whenever $\lam\le 1$.

\begin{corollary}\label{C:2.6}
Assume that $0<\lambda\le1$, and let 
$(x_j)$ and $f$ be as above.  The following holds:
\begin{equation*}
\lnorm\FT{\Il(f)}\rnorm_{\LtwoT} \le \sqrt{2\pi}\big[1+ 2B^4\big]\|f\|_{L_2(\R)}.
\end{equation*}
\end{corollary}

The preceding result accomplishes the first half
of what we set about to do in this section.
The second part will be dealt with next; our
deliberations will be
quite brief,
for the proof is now familiar terrain.

\begin{proposition}\label{P:2.7}
Let $0<\lam\le1$.  The following holds:
\begin{align*}
\lnorm\FT{\Il(f)}\rnorm_{L_2(\R\setminus[-\pi,\pi])}
\le B^2\sqrt{\frac\pi\lambda}\|\psi_\lam\|_{\LtwoT}\sqrt{
\kappa(\pi^2/(2\lam))}\le\sqrt{8\pi}  B^2\|f\|_{L_2(\R)}.
\end{align*}
\end{proposition}

\begin{proof}
We begin by noting that
\begin{align}\label{E:2.16}
\lnorm\FT{\Il(f)}\rnorm_{L_2(\R\setminus[-\pi,\pi])}^2 &=
\frac{\pi}{\lam}\int_{\R\setminus[-\pi,\pi]}
e^{-u^2/(2\lam)}|\Psi_\lam(u)|^2\,du\\
\notag &=
\frac{\pi}{\lam}\summ\lint{(2l-1)\pi}{(2l+1)\pi}
e^{-u^2/(2\lam)}|\Psi_\lam(u)|^2\,du.
\end{align}
The last term in \eqref{E:2.16} may be bounded as follows:
\begin{align}\label{E:2.17}
\frac{\pi}{\lam}\summ\lint{-\pi}{\pi}e^{-(t+\Lwoodone l)^2/(2\lam)}
&\left|\left(
A_l(\psi_\lam)\right)(t)\right|^2\,dt\\
\notag &\le
\frac{\pi}{\lam}\summ e^{-(2|l|-1)^2\pi^2/(2\lam)}
\lnorm A_l(\psi_\lam)\rnorm_{\LtwoT}^2 \\
\notag &\le \frac{B^4\pi}{\lam}\lnorm\psi_\lam\rnorm_{\LtwoT}^2
\kappa(\pi^2/(2\lam)),
\end{align}
the last inequality being consequent upon \eqref{E:2.2}. This proves the first
 of the two stated inequalities.

The second inequality follows from the first,
by way of Proposition \ref{P:2.5}, \eqref{E:0.1}
and the fact that $(1-e^{-\pi^2/4\lam})^{-1}\le 2$, whenever $\lam\le 1$. 
\end{proof}
We close this section by summarizing the findings of
Corollary \ref{C:2.6} and Proposition \ref{P:2.7}:
\begin{theorem}\label{T:2.8}
Suppose that $(x_j:j\in\Z)$ is a fixed Riesz-basis sequence.
Then $\lbrace\Il:0<\lam\le1\rbrace$ is a uniformly-bounded
family of linear operators from $\PW$ to $\LtwoR$.
\end{theorem}

\section{Convergence of $\Il$}\label{S:3}

This section is devoted to the proof of the convergence
result stated in the introduction (Theorem \ref{T:0.1}).
We begin by laying some requisite groundwork.
Let $\alpha$
be a fixed positive number.
Recall the linear operator $\Top{\alpha}$ 
from \eqref{E:2.13}: 
\begin{equation*}
\Top{\alpha}(h):=e^{\pi^2\alpha}\summ\Alstar\left(
e^{-\alpha(\cdot+2\pi l)^2}A_l(h)
\right), \qquad h\in\LtwoT.
\end{equation*}
We now define the following (multiplier) operator
on $\LtwoT$:
\begin{equation}\label{E:3.0}
\Mop{\alpha}(h):=e^{-\alpha(\pi^2-(\cdot)^2)}h, \qquad h\in\LtwoT.
\end{equation}

The following properties of $\Mop{\alpha}$ are easy
to verify.

\begin{proposition}\label{P:3.0}
The operator $\Mop{\alpha}$ is a bounded
linear operator on $\LtwoT$, whose norm
does not exceed $1$.  Moreover, it is
self adjoint, 
strictly positive, and invertible.
\end{proposition}

In what follows, we let 
$(x_j:j\in\Z)$
be a fixed Riesz-basis sequence, and let
$r:\LtwoR\to\LtwoT$
denote the map which sends a function
in $\LtwoR$ to its restriction to the interval
$[-\pi,\pi]$; note that $r$ is a bounded
linear map with unit norm.  

Let $f\in\PW$.  
Recall (from part (iii)
of Theorem \ref{T:1.6}) that, if
$\Il(f)$ denotes the Gaussian interpolant
to $f$ at the points $(x_j)$, then
\begin{equation}\label{E:3.1}
\psi_\lam(t)=
\sqrt{(\lam/\pi)}\,e^{t^2/(4\lam)}\FT{\Il(f)}(t)
\end{equation}
for almost every $t$ in $[-\pi,\pi]$
(remembering that $\psi_\lam$ is 
the restriction of $\Psi_\lam$ to the
interval $[-\pi,\pi]$).
With all this in mind,
we find from the definitions
of $\Top{\alpha}$ and $\Mop{\alpha}$ that
equation \eqref{E:2.9} may be cast
in the following form:
\begin{equation}\label{E:3.2}
\left(
\cI+\Top{1/(4\lam)}\Mop{1/(4\lam)}\right)
r(\FT{\Il(f)})=r(\FT{f}),
\end{equation}
where $\cI$ denotes the identity 
on $\LtwoT$. 

Suppose that $g\in\LtwoT$, and let
\[{\tilde g}(t)=
\begin{cases}g(t)&\text{if $t\in[-\pi,\pi]$};\\
0&\text{if $t\in\R\setminus[-\pi,\pi]$}.
\end{cases}
\]
Then there
is an $f\in\PW$ such that $\FT{f}={\tilde g}$;
in fact, we may take $f$ to be the following:
\begin{equation*}
f(x):=\frac{1}{\Lwoodone}\lint{-\pi}{\pi}g(t)e^{ixt}\,dt, \qquad x\in\R.
\end{equation*} 
Let $\Il(f)$ denote the interpolant to $f$ at 
the points $(x_j)$, and define $L_\lam(g):=r(\FT{\Il(f)})$;
in other words, $L_\lam=r\circ\cF\circ\Il\circ{\cF}^{-1}$.
As the Fourier transform $\cF$ is a linear
isomorphism on $\LtwoR$, and the maps
$\Il$ and $r$ are linear and continuous, 
the map
$g\mapsto L_\lam(g)$ is a continuous linear operator
on $\LtwoT$; moreover, equation \eqref{E:3.2}
affirms that 
\begin{equation}\label{E:3.3}
\left(
\cI+\Top{1/(4\lam)}\Mop{1/(4\lam)}\right)
L_\lam(g)=g.
\end{equation}
This being true
for every $g\in\LtwoT$, we deduce that the map 
$\cI+\Top{1/(4\lam)}\Mop{1/(4\lam)}$ is 
surjective on $\LtwoT$, and that 
$L_\lam$ is a right inverse of 
$\cI+\Top{1/(4\lam)}\Mop{1/(4\lam)}$.  We now 
show that $\cI+\Top{1/(4\lam)}\Mop{1/(4\lam)}$
is, in fact, invertible on $\LtwoT$.

\begin{proposition}\label{P:3.1}
The map $\cI+\Top{1/(4\lam)}\Mop{1/(4\lam)}$
is injective, hence invertible, on 
$\LtwoT$.  Moreover,
there is a  constant $\Delta$, 
depending only on the sequence $(x_j)$,
such that
\[
\lnorm\left(
\cI+\Top{1/(4\lam)}\Mop{1/(4\lam)}\right)^{-1}\rnorm\le\Delta,
\qquad 0<\lambda\le1.
\]
\end{proposition}

\begin{proof}
Suppose that 
$\left(\cI+\Top{1/(4\lam)}\Mop{1/(4\lam)}\right)g=0$
for some $g\in\LtwoT$.  Then 
\begin{align*}
0&=\left\langle
\left(\cI+\Top{1/(4\lam)}\Mop{1/(4\lam)}\right)(g),\Mop{1/(4\lam)}(g)\right
\rangle\\
&=
\left\langle g,\Mop{1/(4\lam)}(g)\right
\rangle
+\left\langle
\Top{1/(4\lam)}\Mop{1/(4\lam)}(g),\Mop{1/(4\lam)}(g)\right
\rangle \\
\notag &\ge\left\langle g, \Mop{1/(4\lam)}(g)\right
\rangle\ge0,
\end{align*}
where we have used the positivity of the operators
$\Top{1/(4\lam)}$ and $\Mop{1/(4\lam)}$ to obtain
the two inequalities above.  It follows
that $\Mop{1/(4\lam)}(g)=0$, and,
as $\Mop{1/(4\lam)}$ is 
strictly positive, $g$ must be zero.  Hence 
$\cI+\Top{1/(4\lam)}\Mop{1/(4\lam)}$ is injective,
and therefore invertible.
So \eqref{E:3.3} may now be stated as
follows: 
\[
\left(
\cI+\Top{1/(4\lam)}\Mop{1/(4\lam)}\right)^{-1}=L_\lam.
\]
Consequently, the uniform boundedness of 
$\Big\|\left(
\cI+\Top{1/(4\lam)}\Mop{1/(4\lam)}\right)^{-1}\Big\|
$
for $\lam\in(0,1]$ obtains from recalling
the equation 
$L_\lam=r\circ\cF\circ\Il\circ{\cF}^{-1}$, along with 
Theorem \ref{T:2.8}.
\end{proof}

We are now ready for the first of the two focal
results of this section.

\begin{theorem}\label{T:3.2}
If $f\in\PW$, then 
\[
\lim_{\lam\to0^+}\lnorm f-\Il(f)\rnorm_{\LtwoR}=0.
\]
\end{theorem}

\begin{proof}
In view of the first identity in \eqref{E:0.1},
it is sufficient to show that 
\begin{equation}\label{E:3.4}
\lim_{\lam\to0^+}\lnorm\FT{f}-\FT{\Il(f)}\rnorm_{\LtwoR}=0.
\end{equation}
Assume that $0<\lam\le1$, and let $\lambda':=1/(4\lam)$.
As $\FT{f}$ is zero almost everywhere outside
$[-\pi,\pi]$, we see that
\begin{equation}\label{E:3.5}
\lnorm\FT{f}-\FT{\Il(f)}\rnorm_{\LtwoR}^2=
\lnorm\FT{f}-\FT{\Il(f)}\rnorm_{\LtwoT}^2\!
+\!\lnorm\FT{\Il(f)}\rnorm_{L_2(\R\setminus[-\pi,\pi])}^2.
\end{equation}
On the interval $[-\pi,\pi]$, we have, via
\eqref{E:3.2}, that 
\begin{align*}
\FT{f}-\FT{\Il(f)}&=\Big[
\cI-\Big(\cI+\Top{\lam'}\Mop{\lam'}\Big)^{-1}\Big]\FT{f}\\
&=\Big(\cI+\Top{\lam'}\Mop{\lam'}\Big)^{-1}
\Top{\lam'}\Mop{\lam'}\left(\FT{f}\right),
\end{align*}
where the second step is a matter of direct verification.  Consequently
\begin{align}\label{E:3.6}
\|\FT{f}-&\FT{\Il(f)}\|_{\LtwoT}\\
\notag &\le\lnorm\Big(\cI+\Top{\lam'}\Mop{\lam'}\Big)^{-1}\rnorm
\,\lnorm\Top{\lam'}\rnorm\,\lnorm\Mop{\lam'}\left(
\FT{f}\right)\rnorm_{\LtwoT}.
\end{align}
Now Proposition \ref{P:3.1} provides a positive constant
$\Delta$ which bounds the first term on the right
of \eqref{E:3.6} for every $\lam\in(0,1]$.
As $\kappa(\pi^2/(4\lam))=O(e^{-\pi^2/(4\lam)})$
for $0<\lam\le1$,
Proposition \ref{P:2.4}
implies that 
\[
\lnorm\Top{\lam'}\rnorm=O\left(
e^{\pi^2/(4\lam)}\kappa(\pi^2/(4\lam))\right)=O(1),\qquad 0<\lam\le1,
\]
for some Big-O constant which is independent of $\lam$.
Using this pair of estimates in \eqref{E:3.6} provides
\begin{equation}\label{E:3.7}
\lnorm\FT{f}-\FT{\Il(f)}\rnorm_{\LtwoT}=O\left(
\lnorm\Mop{\lam'}\left(
\FT{f}\right)\rnorm_{\LtwoT}\right), \qquad 0<\lam\le1.
\end{equation}
Turning to the second term on the right of 
\eqref{E:3.5}, we see from 
Proposition \ref{P:2.7},
\eqref{E:3.1}, and \eqref{E:3.0} that
\begin{align}\label{E:3.8}
\lnorm\FT{\Il(f)}\rnorm_{L_2(\R\setminus[-\pi,\pi])}^2&=
O\left(\frac{\lnorm\psi_\lam\rnorm_{\LtwoT}^2\kappa(\pi^2/(2\lam))}{\lam}\right)\\
\notag &=O\left(
\lnorm e^{(\cdot)^2/(4\lam)}\FT{\Il(f)}\rnorm_{\LtwoT}^2\kappa(\pi^2/(2\lam))
\right)\\
\notag &=O\left(
e^{\pi^2/(2\lam)}\lnorm\Mop{\lam'}\left(\FT{\Il(f)}\right)
\rnorm_{\LtwoT}^2\kappa(\pi^2/(2\lam))\right)\\
\notag &=O\left(\lnorm\Mop{\lam'}\left(\FT{\Il(f)}\right)
\rnorm_{\LtwoT}^2\right),\qquad 0<\lam\le1,
\end{align}
the final step resulting from the (oft cited)
estimate $\kappa(\pi^2/(2\lam))=O(e^{-\pi^2/(2\lam)})$,
$0<\lam\le1$.  Now
\begin{align}\label{E:3.9}
&\lnorm\Mop{\lam'}\left(\FT{\Il(f)}\right)
\rnorm_{\LtwoT}^2\\
\notag &=O\left(
\lnorm\Mop{\lam'}\left(\FT{\Il(f)}-\FT{f}\right)\rnorm_{\LtwoT}^2
+\lnorm\Mop{\lam'}\left(\FT{f}\right)\rnorm_{\LtwoT}^2\right)\\
\notag &=O\left(\lnorm\Mop{\lam'}\left(\FT{f}\right)\rnorm_{\LtwoT}^2\right),
\end{align}
because $\|\Mop{\lam'}\|\le1$ (Proposition \ref{P:3.0}),
and \eqref{E:3.7} holds.
Combining \eqref{E:3.9} and \eqref{E:3.7}
with \eqref{E:3.5}, we find that
\begin{equation}\label{E:3.9b}
\lnorm\FT{f}-\FT{\Il(f)}\rnorm_{\LtwoR}^2=
O\left(\lnorm\Mop{1/(4\lam)}\left(\FT{f}\right)
\rnorm_{\LtwoT}^2\right)=o(1),\quad \lam\to0^+,
\end{equation}
the last assertion being a consequence of
the Dominated Convergence Theorem.  This
establishes \eqref{E:3.4}, and the proof is complete.
\end{proof}

The final theorem of the section deals
with uniform convergence.

\begin{theorem}\label{T:3.3}
If $f\in\PW$, then $\lim_{\lam\to0^+}\Il(f)(x)=f(x)$,
$x\in\R$, and the convergence is
uniform on $\R$.  
In particular,  
the operators
$I_\lambda$, $0<\lambda \le 1$, are uniformly bounded as  operators
from $\PW$ to $C_0(\R)$, via the Uniform Boundedness Principle.  
\end{theorem}

\begin{proof}
Assume that $0<\lam\le1$, and let $x\in\R$.
From \eqref{E:0.5}, part (iv)
of Theorem \ref{T:1.6}, and the fact
that $\FT{f}=0$ almost everywhere on
$\R\setminus[-\pi,\pi]$, we see that
\begin{align}\label{E:3.10}
f(x)-\Il(f)(x)&=\frac{1}{2\pi}\lint{-\pi}{\pi}(
\FT{f}(u)-\FT{\Il(f)}(u))e^{ixu}\,du\\
\notag &\qquad\qquad
-\frac{1}{2\pi}\int_{\R\setminus[-\pi,\pi]}\FT{\Il(f)}(u)e^{ixu}\,du.
\end{align}
The modulus of the first term on the right 
side of \eqref{E:3.10} is no larger than
\begin{equation}\label{E:3.10b}
\lnorm\FT{f}-\FT{\Il(f)}\rnorm_{L_1[-\pi,\pi]}=O\big(
\lnorm\FT{f}\!-\!\FT{\Il(f)}\rnorm_{L_2[-\pi,\pi]}\big)=o(1),\,\lambda\to0^+,
\end{equation}
via \eqref{E:0.1} and Theorem \ref{T:3.2}.  The second term on the right
side of \eqref{E:3.10} is estimated in a familiar way:
\begin{align}\label{E:3.11}
\Big|\int_{\R\setminus[-\pi,\pi]}&\FT{\Il(f)}(u)e^{ixu}\,du
\Big|\\
\notag &\le \int_{\R\setminus[-\pi,\pi]}\left|\FT{\Il(f)}(u)\right|\,du\\
\notag &=\sqrt{\frac{\pi}{\lam}}\,\summ\lint{(2l-1)\pi}{(2l+1)\pi}
e^{-u^2/(4\lam)}|\Psi_\lam(u)|\,du\\
\notag &\le \sqrt{\frac{\pi}{\lam}}\,\summ
e^{-(2|l|-1)^2\pi^2/(4\lam)}\lnorm A_l(\psi_\lam)\rnorm_{L_1[-\pi,\pi]}\\
\notag &=O\left(\frac{1}{\sqrt\lam}\,\summ
e^{-(2|l|-1)^2\pi^2/(4\lam)}\lnorm A_l(\psi_\lam)\rnorm_{L_2[-\pi,\pi]}\right)\\
\notag &=O\left(\frac{
\lnorm\psi_\lam\rnorm_{L_2[-\pi,\pi]}\kappa(\pi^2/(4\lam))}{\sqrt\lam}
\right).
\end{align}
Borrowing the argument which led up to
\eqref{E:3.8}, \eqref{E:3.9}, and the final
conclusion of \eqref{E:3.9b}, we deduce 
that $I_\lambda(f)(x)$ converges  to $f(x)$ uniformly in $\R$.

\end{proof}

\section{A Multidimensional Extension}\label{S:4}

We  consider now the  multidimensional Gaussian interpolation operator.
Let $d\in\N$, and let $(x_j:j\in\N)\subset \R^d$. 
 We say that $(x_{j}:j\in\N)=
(x_{(j,1)}, x_{(j,2)},\ldots,x_{(j,d)}:j\in\N)\subset \R^d$ is a 
{\em $d$-dimensional Riesz-basis sequence}
 if the sequence $(e^{(j)}:j\in\N)$, with 
$$
e_j:[-\pi,\pi]^d\to \mathbb C,\quad 
e_j(t_1,t_2,\ldots, t_d):=e^{-i\langle x_j, t\rangle}=
e^{-i \sum_{l=1}^d x_{(j,l)} t_l },
$$
is a Riesz basis of $L_2[-\pi,\pi]^d$.

In general there is no {\em natural  indexing } of the elements of $\{x_j\}$ by 
$\Z$ or $\Z^d$ if $d\ge 2$; so 
we index generic $d$-dimensional Riesz-basis sequences by $\N$.
Later we shall concentrate on {\em grids } in $\R^d$, {\it i.e.,\/} on 
Riesz bases in $\R^d$ which are products of one-dimensional 
Riesz-basis sequences. In that case 
the natural index set is $\Z^d$.  We note that,
as in the 1-dimensional case, a Riesz 
basis-sequence in $\R^d$ also has to be separated \cite{LM}.

The {\em $d$-dimensional Gaussian function with scaling parameter} 
$\lambda>0$ is defined by
$$
g^{(d)}_\lambda(x_1,x_2,\ldots,x_d)=e^{-\lambda\|x\|^2}
= e^{-\lambda\sum_{j=1}^d x_j^2},
\quad x=(x_1,x_2,\ldots, x_d)\in \R^d.
$$ 

The {\em Fourier transform} on  $L_1(\R^d)$ and 
that in $L_2(\R^d)$  are  defined  as in the 1-dimensional
case, and  we denote  it   by  $\hat g$, if  $g\in L_1(\R^d)$, and
by $\FTd{g}$, if   $g\in L_2(\R^d)$.

The Paley-Wiener functions on $\R^d$ are given by
$$
\PW^{(d)}:=\{ g\in  L_2(\R^d): \FTd{g}=0 \text{ almost everywhere outside $[-\pi,\pi]^d$}\}.
$$

The following multidimensional result can be shown in the same way as 
the corresponding result in the 1-dimensional case (see Theorem  \ref{T:1.6}). 
Its proof is omitted.

\begin{theorem}\label{T:4.1}
 Let $d\in\N$,  let $\lambda$ be  a fixed postive number, 
and let $(x_j:j\in\N)$ be a Riesz-basis sequence
 in $\R^d$.
 For any $f\in \PW^{(d)}$ there exists a unique square-summable 
sequence $(a(j,\lambda):j\in\N)$
such that 
\begin{equation}
 \label{E:4.1.1} \sum_{j\in\N} a(j,\lambda) g^{(d)}_\lambda(x_k-x_j)=f(x_k), \quad k\in\N.
\end{equation}
The {\em Gaussian Interpolation Operator}
$I^{(d)}_\lambda: \PW^{(d)}\to L_2[-\pi,\pi]^d$,
defined by
$$I^{(d)}_\lam(f)(\cdot)=\sum_{j\in\N} a(j,\lambda) g^{(d)}_\lambda(\cdot\!-\!x_j),$$
where $(a(j,\lambda):j\in\N)$ satisfies \eqref{E:4.1.1},
is a well-defined, bounded linear operator from $\PW^{(d)}$ to  $L_2[-\pi,\pi]^d$.
Moreover, $I^{(d)}_\lam(f)\in C_0(\R^d)$.
\end{theorem}

We now  generalize Theorems \ref{T:3.2} and \ref{T:3.3} to the multidimensional
 case assuming that the underlying Riesz-basis sequence is a grid. For simplicity
 we restrict ourselves to the case when $d=2$, but note that our arguments
 carry over readily to all other values of $d$.  
It ought to be noted, however, that even this
simplistic situation, namely that our data
sites form a grid, cannot be handled via a straightforward 
multivariate extension
of the crucial ingredients from Section 3.  In particular, the proofs 
of Corollary 3.7 and Proposition 3.8 do not extend to higher
dimensions.  So we pursue a different tack below.    

We recall some basic tools 
from tensor products. Let $X$ and $Y$ be two Banach
 spaces and $X^*$ and $Y^*$ their  dual spaces.
The algebraic tensor product of $X$ and $Y$ is denoted by
$X\!\otimes\! Y$, and consists of the vector space 
of all linear combinations of elementary products
 $x\!\otimes\! y$ with $x\in X$ and $y\in Y$.
 All our Banach spaces are spaces of functions, and we can therefore identify
  elements of $X\!\otimes\! Y$ with functions on a product of sets. 
 
A norm $\alpha(\cdot)$ on $X\!\otimes\! Y$ is called 
a {\em reasonable cross norm }
 if $\alpha(x\!\otimes\! y)\le \|x\|\cdot\|y\|$, 
and if for $\phi\in X^*$ and $\psi\in Y^*$, the
map
$\phi\!\otimes\! \psi: X\!\otimes\! Y\to C,\quad \sum x_i\!\otimes\! y_i\mapsto \sum \phi(x_i)\psi(y_i)$,
is bounded on $(X\!\otimes\! Y,\alpha)$, and
$\alpha(\phi\!\otimes\! \psi)=\sup_{u\in X\!\otimes\! Y, \alpha(u)\le 1} |\phi\!\otimes\! \psi(u)|\le \|\phi\|\cdot\|\psi\|$.
If $\alpha(\cdot)$ is a norm on $X\!\otimes\! Y$, we denote the completion
 of $X\!\otimes\! Y$ with respect to $\alpha(\cdot)$ by $X\!\otimes\!_\alpha Y$.

\begin{proposition}\label{P:4.2}{cf. \rm\cite[Section  and 3.1, Proposition 6.1]{Ry}}\\
Let $X$ and $Y$ be Banach spaces. For $u\in X\!\otimes\! Y$ define
\begin{align*}
\vp(u)&:=\sup\{ \phi\!\otimes\! \psi(u): \phi\in X^*,\,\psi\in Y^*, \|\phi\|,\|\psi\|\le 1\}\\ 
\end{align*}
\item{(i)} $\vp(\cdot)$ is a reasonable cross norm, 
and  we call $\vp$  the {\em injective tensor norm on $X\!\otimes\! Y$}.
\item{(ii)} If $\alpha(\cdot)$ is any reasonable cross norm on $X\times Y$, then
$$
\vp(u)\le \alpha(u), \quad u\in X\!\otimes\! Y.
$$
\item{(iii)} For Banach spaces $V$ and $W$, and bounded operators
 $S:X\to U$ and $T:Y\to W$, the map
$S \!\otimes\!T: X\!\otimes\! Y\to V\!\otimes\! W$,
defined by
$S\!\otimes\!T(\sum_{i=1}^n x_i\!\otimes\! y_i) =  
   \sum_{i=1}^n S(x_i)\!\otimes\! T(y_i)$, 
extends (uniquely, by density) to a bounded operator 
$S\!\otimes_\vp\! T:  X\!\otimes_\vp\! Y\to V\!\otimes_\vp\! W$, and
 $\|S\!\otimes_\vp\! T\|= \|S\|\cdot \|T\|$.
\end{proposition}
On tensor products of Hilbert spaces we can define a  unique tensor norm for which the completions
 are again Hilbert spaces.
\begin{proposition}\label{P:4.3}{\rm cf. \cite[p. 125 ff]{KR} }\\
Let $H$ and $K$ be Hilbert spaces.
\item{(i)} There 
 is a unique inner product $\langle\cdot,\cdot\rangle_{H\!\otimes\! K}$  
on $H\!\otimes\! K$
for which
$$
\qquad\qquad\langle h_1\!\otimes\!k_1, h_2\!\otimes\!k_2\rangle_{H\!\otimes\! K}=
 \langle h_1, h_2\rangle_{H} \langle k_1, k_2\rangle_{K},\,
 h_1,h_2\!\in\!H\text{ and }k_1,k_2\!\in\!K,
$$
where $\langle\cdot,\cdot \rangle_H$ and  $\langle\cdot,\cdot \rangle_K$ denote
the inner products in $H$ and $K$, respectively. We denote the completion
 of $H\!\otimes\! K$ with respect to the
 corresponding Hilbert norm $\|\cdot\|_h$ 
({\it i.e.,\/} $\|u\|_h= \langle u, u\rangle_{H\!\otimes\! K}^{1/2}$, 
for $u\in H\!\otimes\! K$) by $H\otimes_{h} K$.
\item{(ii)} $\|\cdot\|_h$ is a reasonable cross norm on $H\!\otimes\! K$.
\item{(iii)} If $V$ and $W$ are two Hilbert spaces, and $S:H\to V$ and $T:K: W$ 
are two bounded operators
then the map
$S\!\otimes\!T: H\!\otimes\! K\to V\!\otimes\! W$, defined via 
$S\!\otimes\!T(\sum_{i=1}^n x_i\!\otimes\! y_i) =  \sum_{i=1}^n S(x_i)\!\otimes\! T(y_i)$,
extends to a bounded operator $S\!\otimes_h\! T:  X\!\otimes_h\! Y\to V\!\otimes_h\! W$, and
 $\|S\otimes_h\! T\|= \|S\|\cdot \|T\|.$
\item
  Let $(e_i:i\in\N)\subset H$ and $(f_i:i\in\N)\subset K$. Then
  $(e_i:i\in\N)$ and $(f_i:i\in\N)$  are Riesz bases, or orthogonal bases for $H$ and $K$, respectively,
  if and only if $(e_i\!\otimes\! f_j:i,j\in\N)$ is a Riesz basis, or an orthogonal basis, of $H\!\otimes_h\! K$. 
\item If $(\Omega,\Sigma, \mu)$ and $(\Omega',\Sigma', \mu')$ are two measure spaces, and $H=L_2(\mu)$, and
$K=L_2(\mu')$, then $H\!\otimes_h\! K$ is, via the identifcation of $f\otimes g$ with
 the function $\Omega\times \Omega'\ni(\omega,\omega')\mapsto g(\omega)f(\omega')$, isometrically
 isomorphic to $L_2(\mu\times \mu')$.
\end{proposition}
Using 
Propositon \ref{P:4.3} and the identification of $L_2(\R)\!\otimes\! L_2(\R)$ with 
 $L_2(\R^2)$, we deduce that $\FTt{\cdot}= \FT{\cdot}\!\otimes_h\!\FT{\cdot}$.
 Consequently,
\begin{align*}
\PWt=(\cF^{(2)})^{-1}(L_2[-\pi,\pi]^2)&=
(\cF\!\otimes_h\cF)^{-1}\big(L_2[-\pi,\pi]\!\otimes_h\!L_2[-\pi,\pi]\big)\\
 &= \PW\!\otimes_h\!\PW.
\end{align*}

For the remainder of this section 
we fix two Riesz-basis sequences $(x_j:j\in\Z)$ and $(y_j:j\in\Z)$, 
and we put
$z_{(l,k)}= (x_{l},y_k)$ for $l,k\in\Z$. By Proposition \ref{P:4.3},
$(z_{(l,k)}:l,k\in\Z)$ is a Riesz-basis sequence on $\R^2$, and 
we denote the  1-dimensional Gaussian interpolation operator 
corresponding to $\lam>0$ and 
$(x_j:j\in\Z)$ by $\Il$, and the one associated to
$\lam>0$
and $(y_j:j\in\Z)$ by $\Il'$.
The  2-dimensional Gaussian interpolation operator 
corresponding to  $(z_{(l,k)}:l,k\in\Z)$ is denoted  
 by $I^{(2)}_\lambda$.
\begin{proposition}\label{P:4.4} For any positive number $\lambda$,  
$$
I^{(2)}_\lambda=\Il\!\otimes_h\!\Il'.
$$
\end{proposition}
\begin{proof}
From Theorem \ref{T:1.6}, Theorem \ref{T:4.1}, and Proposition
 \ref{P:4.3}, we find that the 
operators $\Il\!\otimes_h\! \Il'$ and
 $I^{(2)}_\lambda$ are bounded linear operators 
from $\PWt$ into $L_2[-\pi,\pi]^2$.  So we only need to verify that,
 $$
I^{(2)}_\lambda(f\!\otimes h)=\Il(f)\!\otimes_h\! \Il'(h),\qquad 
f,h\in\PW.
$$
Recall from Theorem \ref{T:1.6} that there exists
a unique pair of square-summable sequences
$(a(j,\lambda):j\in\Z)$ and $(b(j,\lambda):j\in\Z)$
such that
$$
\sum_{j\in\Z} a(j,\lambda) g_\lambda(x_k -x_j)=f(x_k)
\text{ and }
\sum_{j\in\Z} b(j,\lambda) g_\lambda(y_k-y_j)=h(y_k),\quad k\in\Z.
$$
This yields for $k,l\in\Z$ 
\begin{align*}
f(x_k)h(y_l)&=\sum_{j\in\Z} a(j,\lambda) g_\lambda(x_k -x_j)
\sum_{j\in\Z} b(j,\lambda) g_\lambda(y_l-y_j)\\
&=\sum_{j,m\in\Z}a(j,\lambda) b(m,\lambda) g_\lambda(x_k -x_j)g_\lambda(y_l -y_m)\\
&=\sum_{j,m\in\Z}a(j,\lambda) b(m,\lambda) g^{(2)}_\lambda(z_{(k,l)}-z_{(j,m)}).
\end{align*}
Therefore
the uniqueness in  Theorem \ref{T:4.1} implies that, 
for every $(x,y)\in \R^2$, 
\begin{align*}
I^{(2)}(f\!\otimes h)(x,y)&=\sum_{j,m\in\Z}a(j,\lambda) b(m,\lambda) g^{(2)}_\lambda ((x,y)-(x_j,y_m))\\
                        &=\sum_{j,m\in\Z}a(j,\lambda) b(m,\lambda) g_\lambda (x-x_j) g_\lambda(y-y_m)=
 \Il(f)(x)\Il'(h)(y),
\end{align*}
and this finishes the proof.
\end{proof}
Proposition \ref{P:4.4} allows us to transfer Theorems \ref{T:3.2} and \ref{T:3.3} to the multidimensional case.

\begin{theorem}\label{T:4.5} Suppose that $F\in \PWt$. Then
$\lim\limits_{\lam\to0^+}\|I^{(2)}_\lam(F)-F\|_{L_2(\R^2)}=0$, 
and
$\lim\limits_{\lambda\to 0} I^{(2)}_\lambda(F)(z)=F(z)$
uniformly for $z\in\R^2$.
\end{theorem}
\begin{proof}  Let $F\in\PWt$. The first assertion follows from
 the aforementioned
fact that $\PW\!\otimes\! \PW$ is dense in  $\PWt$, Proposition \ref{P:4.3},
 and 
 Theorem \ref{T:3.2}. 

In order to show the second statement we first note (cf. \cite[page 50]{Ry}) that the injective
tensor product $C_0(\R)\!\otimes_\vp\! C_0(\R)$ is, via the natural map, isometrically isomorphic to 
 $C_0(\R^2)$. 
By  Theorem \ref{T:3.3} and Proposition  \ref{P:4.2}  (iii) the operators
$$
\Il\!\otimes_\vp\!\Il':  \PW\!\otimes_\vp\!\PW\to C_0(\R^2)
$$
are uniformly bounded. 
 So Propositions \ref{P:4.2}(ii)  and
\ref{P:4.3}(ii) imply that    
$\Il\otimes_h\Il'=I_\lambda^{(2)}$ are also uniformly bounded  operators 
from $\PWt$ to  $C_0(\R^2)$. 
 As $\PW\!\otimes\! \PW$ is dense in $\PW^{(d)}$, we 
argue as in the proof of the first statement, and 
conclude from 
 Theorem \ref{T:3.3}
 that  $I_\lambda^{(2)}(F)$ converges 
 uniformly to $F$, for $F\in\PWt$.
\end{proof}

\section{Further Results on univariate Gaussian interpolation}\label{S:5}

In this final section we return to univariate interpolation, in order
to discuss extensions of some results obtained in Section 2.  We begin
with a general result concerning bi-infinite matrices which
appears to be folkloric.  We have seen two articles which cite
a well-known treatise for it, but our search of the latter
came up emptyhanded.  A proof of the said result is 
indicated in \cite{Ja}, but for sake of completeness
and record, we include a fairly self-contained
and expanded rendition of this argument here.

\begin{theorem}\label{T:5.0} Suppose that
$(A(j,k))_{j,k\in\Z}$ is a bi-infinite
matrix which, as an operator on 
$\ell^2(\Z)$, is self adjoint, positive 
 and invertible.

  Assume further
that there exist positive constants
$\kappa$ and $\gamma$ such that 
$|A(j,k)|\le \kappa e^{-\gamma|j-k|}$
for every pair of integers $j$ and $k$.  
Then there exist constants 
$\tilde{\kappa}$ and $\tilde\gamma$
such that $|A^{-1}(s,t)|\le\tilde\kappa
e^{-\tilde\gamma|s-t|}$ for every
$s,t\in\Z$.
\end{theorem}
For the proof of Theorem \ref{T:5.0} we shall need 
the following pair of lemmata.

\begin{lemma}\label{L:5.0}  Let $(H,\langle\cdot,\cdot\rangle)$ be
a Hilbert space, and let $A:H\to H$ be a bounded 
linear, self-adjoint, positive and invertible  operator.
Let $R:=I-\frac{A}{\|A\|}$, where $I$ denotes
the identity.  Then $R=R^*$, $\langle x,Rx \rangle\ge0$
for every $x\in H$, and $\|R\|<1$.
\end{lemma} 

\begin{remark}
 It is a deep fact (cf.\cite[Proposition 2.4.6]{KR})
 that if  $A:H\to H$ is a bounded linear operator such 
 that $\langle x,Ax\rangle\in\R$ for all $x\in H$, then $A$ is self-adjoint.
 Thus  
a bounded linear operator  $A:H\to H$ 
is self adjoint, positive, and invertible if and only if
$$ \inf_{x\in H, \|x\|=1}\langle x,Ax\rangle>0.$$
\end{remark}

\begin{proof}[Proof of Lemma \ref{L:5.0}] The symmetry of $R$ is evident.  If
$\|x\|=1$, then 
\[
\langle x,Rx \rangle=\|x\|^2-\Big\langle x,\frac{A}{\|A\|}x\Big\rangle.
\]
By the assumptions on $A$ and the BCS inequality we see that 
the term on the right of the preceding equation is between
$0$ and $1$.  Therefore $\|R\|=\sup\{\langle x,Rx \rangle : \|x\|=1\}\le1$.
If
$\|R\|=1$, then there is a sequence $(x_n:n\in\N)$ such that
$\|x_n\|=1$ for every $n$, and 
\[
1=\lim_{n\to\infty}\langle x_n, Rx_n \rangle=\lim_{n\to\infty}
\left(1-\Big\langle x_n,\frac{A}{\|A\|}x_n\Big\rangle\right).
\]
which contradicts the invertibility of $A$.
\end{proof}

\begin{lemma}\label{L:5.1} Suppose that
$(R(s,t))_{s,t\in\Z}$ is a bi-infinite matrix
satisfying the following
condition: there exist positive constants
$C$ and $\gamma$ such that $|R(s,t)|\le
 Ce^{-\gamma|s-t|}$ for every pair
of integers $s$ and $t$.  Given $0<\gamma'<\gamma$,
there is a constant $C(\gamma,\gamma')$,
depending on $\gamma$ and $\gamma'$, such that
$|R^n(s,t)|\le C^n C(\gamma,\gamma')^{n-1}
e^{-\gamma'|s-t|}$ for every $s,t\in\Z$.
\end{lemma} 

\begin{proof} Suppose firstly that $s\neq t\in\Z$, and 
assume without loss that $s<t$.  Note that
\begin{align}\label{E:5.0}
\sum_{u=-\infty}^\infty e^{-\gamma|s-u|}e^{-\gamma'|t-u|}&=
\sum_{u=s}^t e^{-\gamma(u-s)}e^{-\gamma'(t-u)}\\
 \notag&\quad+\!
\sum_{u=-\infty}^{s-1} e^{-\gamma(s-u)}e^{-\gamma'(t-u)}\!+\!\!
\sum_{u=t+1}^\infty e^{-\gamma(u-s)}e^{-\gamma'(u-t)}\\
\notag&=:\Sigma_1+\Sigma_2+\Sigma_3.
\end{align}
Now
\begin{align}\label{E:5.1}
\Sigma_1&=
e^{-\gamma'(t-s)}e^{(\gamma-\gamma')s}\sum_{u=s}^t
e^{-u(\gamma-\gamma')}\\
\notag&=
e^{-\gamma'(t-s)}e^{(\gamma-\gamma')s}\sum_{v=0}^{t-s}e^{-s(\gamma-\gamma')-v(\gamma-\gamma')}\\
\notag&=
e^{-\gamma'(t-s)}\sum_{v=0}^{t-s}e^{-v(\gamma-\gamma')}\le
\frac{e^{-\gamma'(t-s)}}{1-e^{-(\gamma-\gamma')}}.
\end{align}
Moreover,
\begin{align}\label{E:5.2}
\Sigma_2&=
\sum_{v=1}^\infty e^{-\gamma v}e^{-\gamma'(t-s+v)}=
e^{-\gamma'(t-s)}\sum_{v=1}^\infty e^{-v(\gamma+\gamma')}\le
\frac{e^{-\gamma'(t-s)}}{1-e^{-(\gamma+\gamma')}},
\end{align}
whereas
\begin{align}\label{E:5.3}
\Sigma_3=
\sum_{v=1}^\infty e^{-\gamma'v}e^{-\gamma(v+t-s)}
=
e^{-\gamma(t-s)}\sum_{v=1}^\infty e^{-(\gamma+\gamma')v}\le
\frac{e^{-\gamma'(t-s)}}{1-e^{-(\gamma+\gamma')}}.
\end{align}
If $s=t$, then 
\begin{equation}\label{E:5.4}
\sum_{u=-\infty}^\infty e^{-\gamma|s-u|}e^{-\gamma'|t-u|}=
\sum_{u=-\infty}^\infty e^{-(\gamma+\gamma')|s-u|}\le
\frac{2}{1-e^{-(\gamma+\gamma')}}.
\end{equation}
From \eqref{E:5.0}-\eqref{E:5.4} we conclude that
\[
|R^2(s,t)|\le C^2\Big[
\frac{1}{1-e^{-(\gamma-\gamma')}}+\frac{2}{1-e^{-(\gamma+\gamma')}}
\Big]=:C^2 C(\gamma,\gamma').
\]
The general result follows from this via induction.
\end{proof}

\begin{proof}[Proof of Theorem \ref{T:4.5}]    Let
$R=I-\frac{A}{\|A\|}$ be the matrix given in that lemma.
As 
\[
R(j,k)=\frac{A(j,k)}{\|A\|}\;
\hbox{if}\;j\neq k,\quad
\hbox{and}\quad
R(k,k)=\frac{\|A\|-A(k,k)}{\|A\|},
\]
there is some constant $C$ such that 
$|R(s,t)|\le C e^{-\gamma|s-t|}$
for every pair of integers $s$ and $t$.  
As $A=\|A\|\,(I-R)$, and 
$r:=\|R\|<1$ (Lemma \ref{L:5.0}), 
the standard Neumann series expansion yields
the relations
\begin{align}\label{E:5.5}
A^{-1}&=
\|A\|^{-1}\sum_{n=0}^\infty R^n\\
\notag&=
\|A\|^{-1}\sum_{n=0}^{N-1} R^n+\|A\|^{-1}R^N\sum_{n=0}^\infty R^n
=\|A\|^{-1}\sum_{n=0}^{N-1}R^n+R^N A^{-1},
\end{align}
for any positive integer $N$.
As $R^0(s,t)=I(s,t)=0$ if $s\neq t$, $s,t\in\Z$,
we
see from \eqref{E:5.5} that
\begin{equation}\label{E:5.6}
A^{-1}(s,t)=\|A\|^{-1}\sum_{n=1}^{N-1}R^n(s,t)
+[R^N A^{-1}](s,t), \quad s\neq t.
\end{equation}
Choose and fix a positive number $\gamma'<\gamma$,
and recall from 
Lemma \ref{L:5.1} that there is
a constant $C(\gamma,\gamma')$ such that
$|R^n(s,t)|\le C^nC(\gamma,\gamma')^{n-1}
e^{-\gamma'|s-t|}$ for every positive integer
$n$, and every pair of integers $s$ and $t$.
So we may assume that there is 
some constant $D:=D(\gamma,\gamma')>1$ such that
$|R^n(s,t)|\le D^{n}
e^{-\gamma'|s-t|}$ for every positive integer
$n$, and every pair of integers $s$ and $t$.
Using this bound in 
\eqref{E:5.6} provides the 
following estimate for every $s\neq t$:
\begin{align}\label{E:5.7}
|A^{-1}(s,t)|&\le
\|A\|^{-1}e^{-\gamma'|s-t|}\sum_{n=1}^{N-1}D^n+\|A^{-1}\|r^N\\
\notag&\le
\|A\|^{-1}e^{-\gamma'|s-t|}\frac{D^N}{D-1}+\|A^{-1}\|r^N.
\end{align}
Let $m$ be a positive integer such that
$e^{-\gamma'}D^{1/m}<1$, and let
$s,t\in\Z$ with $|s-t|\ge m$.
Writing $|s-t|=Nm+k$, $0\le k\le m-1$, 
we find that 
\begin{equation}\label{E:5.8}
e^{-\gamma'|s-t|}D^N=
\Big[e^{-\gamma'}D^{\frac{1}{m+(k/N)}}\Big]^{|s-t|}\le
[e^{-\gamma'}D^{1/m}]^{|s-t|}.
\end{equation}
Further,
\begin{equation}\label{E:5.9}
r^N=\Big[r^{\frac{1}{m+(k/N)}}\Big]^{|s-t|}\le
[r^{1/2m}]^{|s-t|},
\end{equation}
and combining \eqref{E:5.8} and \eqref{E:5.9} with
\eqref{E:5.7} leads to the following 
bounds for every $|s-t|\ge m$ and an appropriately chosen $\tilde\gamma>0$:
\begin{equation}\label{E:5.10}
|A^{-1}(s,t)|\le
\frac{\|A\|^{-1}}{D-1}[e^{-\gamma'}D^{1/m}]^{|s-t|}
+\|A^{-1}\|[r^{1/2m}]^{|s-t|}=O(e^{-\tilde\gamma|s-t|}).
\end{equation}
On the other hand, if $|s-t|<m$, we obtain
\begin{equation}\label{E:5.11}
|A^{-1}(s,t)|\le\|A^{-1}\|\le
(\|A^{-1}\| e^{m\tilde\gamma})e^{-\tilde\gamma|s-t|},
\end{equation}
and combining \eqref{E:5.10} with \eqref{E:5.11}
finishes the proof.
\end{proof}

A direct consequence of the
previous theorem is the following:

\begin{corollary}\label{C:5.1}
Let $\lambda>0$, and let $(x_j:j\in\Z)$
be a sequence of real numbers satisfying
the following condition: there exists
a positive number $q$ such that
$x_{j+1}-x_j\ge q$ for every $j\in \Z$.
Let $A:=A_{\lam}$ be a bi-infinite
matrix whose entries are given by
$A(j,k):=e^{-\lam(x_j-x_k)^2}$,
$j,k\in\Z$.  Then there exist
positive constants $\beta_1$ and 
$\gamma_1$, depending
on $\lam$ and $q$, such that 
$|A^{-1}(s,t)|\le \beta_1 e^{-\gamma_1|s-t|}$,
$s,t\in\Z$.
\end{corollary}
\begin{proof}
 The hypothesis on the $x_j$'s implies that 
$|x_j-x_k|\ge |j-k|q$, for $j,k\in\Z$.
\end{proof}

\begin{remark} The foregoing
result implies, in particular, that 
$A^{-1}$ is a bounded operator
on every $\ell_p(\Z)$, $1\le p\le\infty$.
\end{remark}

We turn now to the second half of this section, in
which we introduce the fundamental functions
for Gaussian interpolation (at scattered
data sites), and set out some of their 
basic properties.  

\begin{theorem}\label{T:5.3}
Let $\lam>0$ be fixed, and let 
$(x_j:j\in\Z)$ be a sequence
of real numbers satisfying the 
following condition: there exists
$q>0$ such that $x_{j+1}-x_j\ge q$
for every integer $j$.  Let 
$A=A_\lam$ be the bi-infinite
matrix whose entries are given by
$A(j,k)=e^{-\lam(x_j-x_k)^2}$,
$j,k\in\Z$.  Given $l\in\Z$,
let the {\em $l$-th fundamental
function} be defined as follows:
\[
L_{l,\lam}(x):=L_l(x):=\sum_{k\in\Z}A^{-1}(k,l)e^{-\lam(x-x_k)^2},
\qquad x\in\R.
\]
The following hold:
\item{(i)} The function $L_l$ is continuous throughout
$\R$.
\item{(ii)} Each $L_l$ obeys the fundamental
interpolatory conditions $L_l(x_m)=\delta_{{}_{lm}}$,
$m\in\Z$.
\item{(iii)} If, in addition, there exists a positive
number $Q$ such that $x_{j+1}-x_j\le Q$
for every $Q$, then there exist positive
constants $\beta_2$ and $\rho$, depending
on $\lam$, $q$, and $Q$ such that 
$|L_l(x)|\le\beta_2 e^{-\rho|x-x_l|}$
for every $x\in\R$ and every $l\in\Z$.
\item{(iv)} Assume that
$(x_j:j\in\Z)$ satisfies
the condition stipulated in (iii).
Let $(b_l:l\in\Z)$
be a sequence satisfying the following
condition: there exists a positive
number $K$ and a positive integer $P$ such that
$|b_l|\le K|l|^P$ for every integer $l$.  Then
the function 
$\R\ni x\mapsto\sum_{l\in\Z}b_lL_l(x)$ is continuous
on $\R$.
\end{theorem}

\begin{proof} (i) As $A^{-1}$ is a bounded
operator on $\ell_\infty(\Z)$,
the sequence $(A^{-1}(k,l):k\in\Z)$ is bounded.
Hence the continuity of $L_l$ follows from
Proposition \ref{P:1.1}.

(ii) Given $m\in \Z$, we have
\[
L_l(x_m)=\sum_{k\in\Z}A^{-1}(k,l)e^{-\lam(x_m-x_k)^2}
=\sum_{k\in\Z}A^{-1}(k,l)A(m,k)=I(m,l)=\delta_{{}_{lm}}.
\]

(iii) The assumption $x_{j+1}-x_j\le Q$ for 
every integer $j$ implies that $|x_k-x_l|\le|k-l|Q$
for every pair of integers $k$ and $l$.  
Therefore Corollary \ref{C:5.1} leads to the bound
$|A^{-1}(k,l)|\le\beta_1 e^{-\gamma_2|x_k-x_l|}$,
$k,l\in\Z$, where $\gamma_2:=\gamma_1/Q$.  
Consequently,
\begin{align}\label{E:5.12}
|L_l(x)|&\le
\beta_1\sum_{k\in\Z}e^{-\gamma_2|x_k-x_l|}e^{-\lam(x-x_k)^2}\\
\notag&\le
\beta_1\sum_{k\in\Z}e^{-\rho|x_k-x_l|}e^{-\rho(x-x_k)^2},
\qquad x\in\R,
\end{align}
where $\rho:=\min\{\lam,\gamma_2\}$.
Therefore
\begin{align}\label{E:5.13}
e^{\rho|x-x_l|}|L_l(x)|&\le
\beta_1\sum_{k\in\Z}e^{\rho[|x-x_l|-|x_k-x_l|]}e^{-\rho(x-x_k)^2}\\
\notag&\le
\beta_1\sum_{k\in\Z}e^{\rho|x-x_k|}e^{-\rho(x-x_k)^2},\qquad x\in\R.
\end{align}
Fix $x\in\R$, and let $s$ be the integer such that
$x_s\le x<x_{s+1}$.  From \eqref{E:5.13} we obtain
\begin{align}\label{E:5.14}
e^{\rho|x-x_l|}|L_l(x)|&\le
\beta_1\sum_{k=s}^{s+1}e^{\rho|x-x_k|}e^{-\rho(x-x_k)^2}\\
\notag&\qquad\qquad+
\beta_1\sum_{k\in{\Z\setminus\{s,s+1\}}}e^{\rho|x-x_k|}e^{-\rho(x-x_k)^2}\\
\notag&\le
2\beta_1e^{\rho Q}+
\beta_1\sum_{k\in{\Z\setminus\{s,s+1\}}}e^{\rho|x-x_k|}e^{-\rho(x-x_k)^2}.
\end{align}
As $mq\le|x-x_{s-m}|\le(m+1)Q$ for every positive integer
$m$, and $(m-1)q\le|x-x_{s+m}|\le mQ$ for every
positive integer $m\ge2$, the final sum on the right
of \eqref{E:5.14} is no larger than 
\[
\beta_1\sum_{m=1}^\infty e^{\rho(m+1)q}e^{-\rho m^2q^2}+
\beta_1\sum_{m=2}^\infty e^{\rho mq}e^{-\rho(m-1)^2q^2}
=:\beta_2.
\]
Combining this with \eqref{E:5.14} and \eqref{E:5.13}
finishes the proof.

(iv) Each summand is continuous by assertion (i).
Let $x\in\R$, and let $x_s\le x<x_{s+1}$.
Assertion (iii) implies that
\begin{align*}
\Big|\sum_{l\in\Z}b_lL_l(x)\Big|&\le
K\beta_2\Big[
\sum_{l=s}^{s+1}|l|^P+
\sum_{l\in{\Z\setminus\{s,s+1\}}}|l|^Pe^{-\rho|x-x_l|}\Big]\\
&\le
K\beta_2\Big[
\sum_{l=s}^{s+1}|l|^P+
\sum_{m=1}^\infty|l|^Pe^{-\rho mq}+\sum_{m=2}^\infty|l|^Pe^{-\rho(m-1)q}\Big].
\end{align*}
It follows that the series $\sum_{l\in\Z}b_lL_l(x)$ is locally
uniformly convergent, whence the stated result follows.
\end{proof}

The counterpart of Part (iii) of the foregoing theorem
for Gaussian cardinal interpolation was obtained
in \cite{RS1}, and its analogue for spline interpolation
was proved in \cite{deB}.

Earlier in this paper we have discussed Gaussian
interpolation operators associated to functions,
specifically to bandlimited functions.  Here our
perspective changes slightly, as we begin
to think of these interpolation operators
acting on sequence spaces.

\begin{theorem}\label{T:5.4}  Let $\lam>0$
be fixed.  Suppose that $(x_j:j\in\Z)$ is
a real sequence satisfying the following
condition: there exist positive numbers
$q$ and $Q$ such that 
$q\le x_{j+1}-x_j\le Q$ for every integer $j$.
Let $A=A_\lam$ be the bi-infinite
matrix whose entries are given by
$A(j,k)=e^{-\lam(x_j-x_k)^2}$,
$j,k\in\Z$.  Given $p\in[1,\infty]$, and 
${\overline y}:=(y_l:l\in\Z)\in\ell_p(\Z)$,
define 
\[
\Il({\overline y},x):=\sum_{k\in\Z}(A^{-1}{\overline y})_k
e^{-\lam(x-x_k)^2}, \qquad x\in\R,
\]
where $(A^{-1}{\overline y})_k$ denotes the
$k$-th component of the sequence $A^{-1}{\overline y}$.
The following hold:
\item{(i)} The function $\R\ni x\mapsto \Il({\overline y},x)$
is continuous on $\R$.
\item{(ii)} If $x$ is any real number, then
\[
\Il({\overline y},x)=\sum_{l\in\Z}y_lL_l(x),
\]
where $(L_l:l\in\Z)$ is the sequence of 
fundamental functions introduced in the preceding
theorem.
\item{(iii)} There is a constant $\beta_3$,
depending on $\lam$, $q$, $Q$, and $p$, such that
\[
\|\Il({\overline y},\cdot)\|_{L_p(\R)}\le
\beta_3\|{\overline y}\|_{\ell_p(\Z)},
\]
for every ${\overline y}\in\ell_p(\Z)$. 
\end{theorem}

\begin{proof}(i) As $A^{-1}$ is a bounded
operator on $\ell_p(\Z)$,
the sequence $A^{-1}{\overline y}$ is bounded.
Hence the continuity of $\Il({\overline y},\cdot)$ follows from
Proposition \ref{P:1.1}.

(ii) Let $x\in\R$. Then
\[
\Il({\overline y},x)=\sum_{k\in\Z}(A^{-1}{\overline y})_k
e^{-\lam(x-x_k)^2}=
\sum_{k\in\Z}\Big[\sum_{l\in\Z}y_lA^{-1}(k,l)\Big]
e^{-\lam(x-x_k)^2}.
\]
The required result obtains by interchanging the order
of summation, which is justified by the following
series of estimates, the first of which
is consequent upon 
Corollary \ref{C:5.1}, and the last of
which follows from Proposition \ref{P:1.1}.
\begin{align*}
\sum_{k\in\Z}\Big[\sum_{l\in\Z}|y_lA^{-1}(k,l)|\Big]
e^{-\lam(x-x_k)^2}&=
O\Big(\sum_{k\in\Z}\Big[\sum_{l\in\Z}e^{-\gamma_1|k-l|}\Big]
e^{-\lam(x-x_k)^2}\Big)\\
&=
O\Big(\sum_{k\in\Z}e^{-\lam(x-x_k)^2}\Big)=O(1).
\end{align*}

(iii) We begin with $p=\infty$.  If
$x_s\le x<x_{s+1}$, assertion 
(ii), the triangle inequality, 
and a now familiar argument
lead to the following:
\begin{align}\label{E:5.15}
|\Il({\overline y},x)|&
=\left|\sum_{l\in\Z}y_lL_l(x)\right|\\
\notag&\le
\beta_2\Big[
|y_s|+|y_{s+1}|+\sum_{m=1}^\infty|y_{s-m}|e^{-\rho mq}+
\sum_{m=2}^\infty|y_{s+m}|e^{-\rho(m-1)q}\Big]\\
\notag&\le
\frac{2\beta_2}{1-e^{-\rho q}}\,\|{\overline y}\|_{\ell_\infty(\Z)}
=:\beta_3\|{\overline y}\|_{\ell_\infty(\Z)}.
\end{align}
It is immediate that $\|\Il({\overline y},\cdot)\|_{L_\infty(\R)}
\le \beta_3\|{\overline y}\|_{\ell_\infty(\Z)}$.

Suppose now that $1\le p<\infty$.  Let $J_s:=[x_s,x_{s+1})$,
$s\in\Z$, and recall from \eqref{E:5.15}
that, if $x\in J_s$, then
\[
|\Il({\overline y},x)|
\le
\beta_2\Big[
|y_s|+|y_{s+1}|+\sum_{m=1}^\infty|y_{s-m}|e^{-\rho mq}+
\sum_{m=2}^\infty|y_{s+m}|e^{-\rho(m-1)q}\Big].
\]
Therefore, as $|x_{s+1}-x_s|\le Q$, we have 
\begin{align}\label{E:5.16}
\int_{J_s}|\Il({\overline y},x)|^p\,dx&\le
Q\beta_2^p
\Big[
|y_s|+|y_{s+1}|\\
\notag&\qquad\qquad+\sum_{m=1}^\infty|y_{s-m}|e^{-\rho mq}+
\sum_{m=2}^\infty|y_{s+m}|e^{-\rho(m-1)q}\Big]^p.
\end{align}
Let ${\overline u}=(u_k:k\in\Z)$ and 
${\overline v}=(v_k:k\in\Z)$ be a pair
of sequences defined by $u_k=|y_k|$,
$k\in\Z$, and 
\[
v_k=
\begin{cases}1&\text{if $k\in\{-1,0\}$};\\
e^{-\rho kq}&\text{if $k\ge1$};\\
e^{\rho(k+1)q}&\text{if $k\le-2$}.
\end{cases}
\]
Then \eqref{E:5.16} may be recast as follows:
\[
\int_{J_s}|\Il({\overline y},x)|^p\,dx\le
Q\beta_2^p
\left|\sum_{m\in\Z}u_{s-m}v_m\right|^p,
\]
so that
\begin{align}\label{E:5.17}
\|\Il({\overline y},\cdot)\|_{L_p(\R)}&=
\left[
\sum_{s\in\Z}\int_{J_s}|\Il({\overline y},x)|^p\,dx
\right]^{1/p}\\
\notag&\le
Q^{1/p}\beta_2
\left[
\sum_{s\in\Z}\left|\sum_{m\in\Z}u_{s-m}v_m\right|^p
\right]^{1/p}.
\end{align}
Now the Generalized Minkowski Inequality
 (cf. \cite[p. 123]{HLP}) implies that
\begin{align}\label{E:5.18}
\left[
\sum_{s\in\Z}\left|\sum_{m\in\Z}u_{s-m}v_m\right|^p
\right]^{1/p}\!\le
\sum_{m\in\Z}|v_m|\,\left[
\sum_{s\in\Z}|u_{s-m}|^p\right]^{1/p}
\!\le
\frac{2}{1-e^{-\rho q}}\,\|{\overline y}\|_{\ell_p(\Z)},
\end{align}
and a combination of \eqref{E:5.18} and 
\eqref{E:5.17} completes the proof.
\end{proof}

We conclude the paper with a few remarks.  Suppose that
$f$ is a bandlimited function, and let
$d_k=f(x_k)$, $k\in\Z$.  
We have seen
in Proposition \ref{P:1.5} that
this sequence ${\overline d}=(d_k:k\in\Z)$
is square summable.  Furthermore,
as observed in the course of the proof
of Theorem \ref{T:1.6}(v), the
Gaussian interpolant $\Il(f)(\cdot)$ studied
earlier coincides with the function
$\Il({\overline d},\cdot)$ introduced in
this section.  Thus the final part
of the previous theorem presents a 
twofold generalization
of the estimate \eqref{E:1.12}: to
values of $p\in[1,\infty]$ other than
$2$, whilst requiring only that 
the underlying set of sampling points
satisfies condition \eqref{E:0.7}.  In
particular we do {\sl not\/} assume in
Theorem \ref{T:5.4}(iii) that 
$(x_j:j\in\Z)$ is a Riesz-basis sequence.  
However, it is not without interest
to note that the main convergence
theorems obtained in Section 4 
do not hold for data sites which
merely satisfy the quasi-uniformity
condition \eqref{E:0.7}. 
For example, let $X:=\Z\setminus\{0\}$, 
and let $f$ be the bandlimited function 
given by $f(x):=\sin(\pi x)/(\pi x)$,
$x\in\R$.  As $f(x_j)=0$
for every $x_j\in X$, 
$\Il(f)$ is identically zero, so
it is manifest that $\Il(f)$
does not converge to $f$.

Counterparts of the result obtained in
part (iii) of the preceding theorem,
for the case when $x_j=j$, may be found
in \cite{RS2} and \cite{RS1}.  However,
those estimates are much more precise in nature.  

Suppose that $(x_j:j\in\Z)$ is a 
strictly increasing sequence
of real numbers satisfying the following
condition: 
\begin{equation}\label{E:Kadec}
|x_j-j|\le D<1/4, \qquad j\in\Z.
\end{equation} 
Then $(x_j)$ 
is a Riesz-basis sequence \cite{Ka}.
Let
$(L_{l,\lam}:l\in\Z)$ be the associated sequence
of fundamental functions defined
in Theorem \ref{T:5.3}.  Define
\[
G(x):=(x-x_0)\prod_{j=1}^\infty\left(
1-\frac{x}{x_j}\right)\left(1-\frac{x}{x_{-j}}\right),
\qquad x\in\R,
\]
and let
\[
G_l(x):=\frac{G(x)}{(x-x_l)G'(x_l)}, \qquad x\in\R,\quad l\in\Z.
\]
It is shown in \cite{Le} 
each $G_l$ 
is a bandlimited function satisfying
the interpolatory conditions
$G_l(x_m)=\delta_{{}_{lm}}$,
$m\in\Z$.  
Thus we find that $\Il(G_l)=L_{l,\lam}$, and
conclude from Theorems \ref{T:3.2} and
\ref{T:3.3} that $\lim_{\lam\to0^+}L_{l,\lam}
=G_l$ in $L_2(\R)$ and uniformly
on $\R$.  As a particular example,
we learn from \cite{Hi} that, if 
$x_0=0$ and $x_j=x_{-j}=j+c^2/j$, $j\ge1$,
$0<|c|<1/2$, then $(x_j)$ fulfills 
\eqref{E:Kadec}, and that the corresponding
function $G$ is given in closed form:
\[
G(x)=x[\cos(\pi(x^2-4c^2)^{1/2})-\cos(\pi x)]/2\sinh(\pi c),
\qquad x\in\R.
\]

\bigskip

\centerline{ACKNOWLEDGMENTS}
\smallskip
The authors take pleasure in thanking David Kerr and Wally Madych for their time and help.

\end{document}